\documentclass[a4paper,10pt]{amsart}

\usepackage{amsmath,amsthm}
\usepackage{amssymb}
\usepackage[arrow, matrix, curve]{xy}
\usepackage{fancyhdr}

\newtheorem{thm}[equation]{Theorem} 
\newtheorem{prop}[equation]{Proposition}

\newtheorem{lemma}[equation]{Lemma}
\theoremstyle{definition}
\newtheorem{defn}[equation]{Definition}

\newtheorem{remark}[equation]{Remark}

\title{Iterative Differential Embedding Problems in positive Characteristic}
\author{Stefan Ernst}
\address{Stefan Ernst \\ RWTH Aachen \\ Lehrstuhl A f\"ur Mathematik \\ Templergraben 55 \\ 52062 Aachen \\ Germany} 
\email{stefan.ernst@matha.rwth-aachen.de} 
\date{June 27, 2011}

\begin{document}

\maketitle

\begin{abstract} 
In this paper, we prove that every iterative differential embedding problem over an algebraic function field in positive characteristic with an algebraically closed field of constants has a proper solution. 
\end{abstract}

\section{Introduction} 

The inverse problem of differential Galois theory asks which linear algebraic groups occur as Galois groups of a Picard-Vessiot extension $E/F$, for a given differential field $F$. We say that a linear algebraic group $\mathcal{G}$ is realizable over $F$, if there exists a Picard-Vessiot extension $E/F$, such that the Galois group $\text{Gal}(E/F)$ is isomorphic to $\mathcal{G}(K)$, where $K$ is the field of constants of $F$. In \cite{MvdP03}, B.H. Matzat and M. van der Put showed that every reduced linear algebraic group is realizable over an algebraic function field of positive characteristic with an algebraically closed field of constants. \\ 
An embedding problem is an intensification of the inverse problem: Given an epimorphism of linear algebraic groups  $\beta$$: \widetilde{\mathcal{G}} \rightarrow \mathcal{G}$ and a realization $E/F$ of $\mathcal{G}$, does there exist a realization $\widetilde{E}/F$ of $\widetilde{\mathcal{G}}$, such that $E$ is a subfield of $\widetilde{E}$? \\ 
If we can give a positive answer to this question, we say the embedding problem has a proper solution. If only a subgroup of $\widetilde{\mathcal{G}}$ is realizable in this way, we just say that the embedding problem is solvable. \\ 
In usual Galois theory the study of embedding problems offers a great deal of information regarding the structure of the absolute Galois group $G_F := \text{Gal}(F^{\text{sep}}/F)$ (here $F$ denotes an arbitrary field). For example, the group $G_F$ is projective, if and only if all embedding problems over $F$ are solvable. Further by Iwasawa's Freiheitssatz \cite{MM99}, Theorem IV.1.12, the group $G_F$ is free (of countable infinite rank) if and only if all embedding problems over $F$ have a proper solution. \\
Now let $F$ again be an iterative differential field. The Picard-Vessiot extensions of $F$ form a neutral Tannakian category \textbf{T}, with fibre functor $\omega$$: \textbf{T} \rightarrow \textbf{Vect}_K$ (for a suitable field $K$). Then by the Main Theorem of Tannakian formalism, the functor $\text{Aut}^{\otimes}(\omega)$ of $K$-algebras is representable by an affine group scheme $\pi_1(\textbf{T})$. This group scheme is called the Tannakian fundamental group of \textbf{T}. For a different fibre functor of \textbf{T}, we get a fundamental group which is isomorhic to $\pi_1(\textbf{T})$ by an inner automorphism, so the fundamental group is independent of the choice of the fibre functor. This fundamental group is an analogon of the absolute Galois group in usual Galois theory, for example the point group of this group scheme can be understood as the automorphism group of a "universal" Picard-Vessiot extension and is a prolinear group (the projective limit of linear algebraic groups). In this way the solvability of embedding problems is a measure of the freeness of $\pi_1(\textbf{T})$. \\ 
In this paper we only consider differential fields with an algebraically closed field of constants. The most important example for us is $F = K(t)$, which is the rational function field in one variable. \\ 
Differential Galois theory in positive characteristic is explained in detail in the preprint \cite{Mat01} from B.H. Matzat. From there we follow the idea laid out in the dissertation of T. Oberlies (\cite{Obe03}), who studied embedding problems in characteristic zero. At the end of this article we obtain the following result: \\

\textbf{Theorem:} \textit{Let $F$ be an algebraic function field in one variable over an algebraically closed field $K$ of positive characteristic. Then every ID-embedding problem in $\emph{\textbf{AffGr}}_K^{\emph{red}}$ over $F$ with reduced kernel has a proper solution.} \\ 

Here $\text{\textbf{AffGr}}_K^{\text{red}}$ denotes the category of reduced linear algebraic groups which are defined over $K$. It is an interesting open question whether this theorem holds in characteristic zero. 
\\
\\
\textbf{Acknowledgements.} I thank B. H. Matzat for giving me the motivation for this article and his invaluable support during the writing. My thanks also go to J. Hartmann and A. Maurischat for helpful comments and suggestions on the paper.

\section{Basics and Notation}

First we give a short introduction into differential Galois theory in positive characteristic, which is based on the preprint \cite{Mat01}. So for more details and explicit proofs the reader is referred to \cite{Mat01}. \\ 
Throughout this article, we assume that all rings and fields have characteristic $p > 0$. Further we assume that $K$ is an algebraically closed field. In the situation of positive characteristic one is confronted with the problem that the usual differentiation, extended to some transcendental extension of a differential field, causes new constants. This problem can be attacked by using iterative derivations, which were introduced for the first time by H. Hasse and F.K. Schmidt \cite{HS37}. \\
For this purpose let $R$ be a commutative ring. A family $\partial^{*} = (\partial^{(k)})_{k \in \mathbb{N}}$ of maps $\partial^{(k)}$$: R \rightarrow R$ is called an \textbf{iterative derivation} of $R$ if 
\begin{itemize} 
	\item[] $\partial^{(0)} = \text{id}_R$,  \ \ \ \ \ \ \ \ \ \ \ \ \ \ \ \ \ \ \ \ \ \ \ \ \ \ \ \ \ \ \ \ \ \  $\partial^{(k)}(a+b) = \partial^{(k)}(a) + \partial^{(k)}(b)$,  
	\item[] $\partial^{(k)}(a \cdot b) = \sum\limits_{i+j=k}{\partial^{(i)}(a)\partial^{(j)}(b)}$, \ \ \ \ \ \ $\partial^{(j)} \circ \partial^{(i)} = \binom{i+j}{i}\partial^{(i+j)}$ 
\end{itemize}
for all $a, b \in R$ and all $i, j, k \in \mathbb{N}$. The tuple $(R,\partial^{*})$ is called an \textbf{iterative differential ring} (ID-ring). An element $c \in R$ is a \textbf{(differential) constant}, if all its iterative derivatives vanish, i.e., $\partial^{(k)}(c) = 0$ for all $k \ge 1$. The set of all constants of $R$ is denoted by $C(R)$. \\ 
We observe that $(\partial^{(k)})^{p} = 0$ for $k > 0$ (this phenomenon is called the \textbf{trivial p-curvature}). Moreover, the iterative derivation $\partial^{*}$ is already determined by all $\partial^{(k)}$, where $k$ is a $p$-power. \\ 
An \textbf{iterative differential field} (ID-field) $F$ is an ID-ring which is a field. In this case the set of constants $C(F)$ is also a field. The most important example of an ID-field is the field of rational functions $K(t)$ together with the iterative derivation $\partial^{(k)}(t^{n}) = \binom{n}{k}t^{n-k}$, which will be denoted by $\partial^{*}_{t}$. Thus the field of constants is $K$. Contrary to ordinary differentiation in positive characteristic we do not obtain new constants here. \\ 
Let $(R,\partial^{*}_{R})$ be an ID-ring and set $R_l := \bigcap\limits_{j<l}\text{ker}(\partial^{(p^j)}_R)$. Then $(R_l,(\partial^{(kp^l)}_R)_{k \in \mathbb{N}})$ is an ID-ring (see \cite{Mat01}, Corollary 1.8). \\
Next we extend the concept of iterative derivations to modules. For this let $(R,\partial^{*}_R)$ be an ID-ring and $M$ be an $R$-module. A family $\partial^{*}_M = (\partial^{(k)}_M)_{k \in \mathbb{N}}$ of maps $\partial^{(k)}_M$$: M \rightarrow M$ satisfying 
\begin{itemize}
	\item[] $\partial^{(0)}_M = \text{id}_M$,  \ \ \ \ \ \ \ \ \ \ \ \ \ \ \ \ \ \ \ \ \ \ \ \ \ \ \ \ \ \ \  $\partial^{(k)}_M(x+y) = \partial^{(k)}_M(x) + \partial^{(k)}_M(y)$,  
	\item[] $\partial^{(k)}_M(a \cdot x) = \sum\limits_{i+j=k}{\partial^{(i)}_R(a)\partial^{(j)}_M(x)}$,  \ \ \  $\partial^{(j)}_M \circ \partial^{(i)}_M = \binom{i+j}{i}\partial^{(i+j)}_M$ 
\end{itemize}
for all $x, y \in M$ and all $a \in R$ is called an \textbf{iterative derivation} of $M$ and the tuple $(M,\partial^{*}_M)$ is called an \textbf{iterative differential module} (ID-module). The $C(R)$-module 
\begin{align*}
	V_M := \bigcap\limits_{k \in \mathbb{N}_+}{\text{ker}(\partial^{(k)}_M)}
\end{align*}
is called the \textbf{solution space} of $M$. We call $M$ a \textbf{trivial iterative differential module} if $M \cong V_M \otimes_{C(R)} R$. \\ 
Let $(M,\partial^{*}_{M})$ and $(N,\partial^{*}_{N})$ be two ID-modules, and let $\varphi \in \text{Hom}_R(M,N)$. Then $\varphi$ is called an \textbf{iterative differential homomorphism} (ID-homomorphism), if 
\begin{align*}
	\varphi \circ \partial^{(k)}_{M} = \partial^{(k)}_{N} \circ \varphi
\end{align*}
for all $k \in \mathbb{N}$. Further, we denote by $\textbf{IDMod}_R$ the category of all finitely generated iterative differential modules over $R$ with ID-homomorphisms as morphisms. \\ 
If $F$ is an ID-field with algebraically closed field of constants $K$, then by \cite{Mat01}, Remark 2.6, the category $\textbf{IDMod}_F$ together with the forgetful functor 
\begin{align*}
	\omega: \textbf{IDMod}_F \rightarrow \textbf{Vect}_K, \ \ \ (M,\partial^{*}_{M}) \mapsto V_{M \otimes_F E}
\end{align*} 
to the category of $K$-vector spaces as fibre functor is a \emph{neutral Tannakian category}, for a suitable extension $E/F$ (see \cite{Del90} for a definition and properties). \\
By \cite{Mat01}, Theorem 2.8, we get the following connection between ID-modules and projective systems. 

\begin{thm} \label{thm proj equ}
Let $(F,\partial^{*}_{F})$ be an ID-field of characteristic $p > 0$. Then the category $\emph{\textbf{IDProj}}_F$ of projective systems $(N_l,\psi_l)_{l \in \mathbb{N}}$ over $F$ with the properties 
\begin{enumerate}
	\item $N_l$ is an $F_l$-vector space of finite dimension and $\psi_l$ is $F_{l+1}$-linear 
	\item each $\psi_l$ uniquely extends to an $F_l$-isomorphism 
	\begin{align*}
		\widetilde{\psi}_l: N_{l+1} \otimes_{F_{l+1}} F_l \rightarrow N_l
	\end{align*}
\end{enumerate} 
is equivalent to the category $\emph{\textbf{IDMod}}_F$. 
\end{thm}
Where $F_l := \bigcap\limits_{j<l}\emph{ker}(\partial^{(p^j)}_F)$, as above. \\ 
Let $F$ be an ID-field with field of constants $K$ and let $M \in \textbf{IDMod}_F$ be an $n$-dimensional ID-module, then we get an associated \textbf{iterative differential equation} (IDE) 
\begin{align*} \tag{$\ast$}
	\partial^{(p^l)}(\emph{\textbf{y}}) = A_l\emph{\textbf{y}}
\end{align*} 
with $A_l = \partial^{(p^l)}_F(D_0 \cdots D_l)(D_0 \cdots D_l)^{-1}$, where $D_l \in \text{GL}_n(F_l)$ are the matrices of $\psi_l$ with respect to a basis $B_l$ of $M_l$. Let $(R,\partial^{*}_{R})$ be an ID-ring with $R \ge F$ and such that $\partial^{*}_{R}$ extends $\partial^{*}_{F}$, i.e., $\partial^{*}_{R}|_F = \partial^{*}_{F}$. A matrix $Y \in \text{GL}_n(R)$ is called a \textbf{fundamental solution matrix} for the IDE ($\ast$) if $\partial^{(p^l)}(Y) = A_lY$ for all $l \in \mathbb{N}$. \\
The ring $R$ is called an \textbf{iterative Picard-Vessiot ring} (IPV-ring) for the IDE ($\ast$) if it satisfies the following conditions: 
\begin{enumerate}
	\item $R$ is a simple ID-ring (i.e., contains no nontrivial proper ID-ideals), 
	\item the IDE has a fundamental solution matrix $Y$ with coefficients in $R$, 
	\item $R$ is generated over $F$ by the coefficients of $Y$ and $\text{det}(Y)^{\text{-1}}$. 
\end{enumerate}

For each ID-module $M$ there exists an IPV-ring $R$, which is unique up to ID-isomorphisms. Each IPV-ring is an integral domain and we can build the quotient field $\text{Quot}(R) = E$ (IPV-field), which contains no new constants, i.e., $C(E) = K$. \\ 
The group $\text{Gal}_{\text{ID}}(E/F) := \text{Aut}_{\text{ID}}(E/F)$ of all ID-automorphisms is called the \textbf{iterative differential Galois group} of the extension $E/F$. This group is isomorphic to the $K$-rational points of an affine group scheme, which is defined over $K$. If $E/F$ is a separable extension, then $\text{Gal}_{\text{ID}}(E/F)$ is a reduced affine group scheme and the converse holds. In this article we consider mostly reduced affine group schemes, which are equivalent (as a category) to the reduced linear algebraic groups. Therefore $\text{Gal}_{\text{ID}}(E/F) \le \text{GL}_{n,K}$ and we denote by $\text{\textbf{AffGr}}_K^{\text{red}}$ the category of all reduced linear algebraic groups, which are defined over $K$, and with homomorphisms of reduced affine group schemes as morphisms. \\ 
By the following theorem we see that the iterative differential Galois theory is a generalization of the usual Galois theory (see \cite{Mat01}, Theorem 3.12). 
\begin{thm} \label{finite gal extension}
Let $F/K$ be an algebraic function field in one variable over $K$. Then the following are equivalent: 
\begin{enumerate}
	\item $E/F$ is a finite Galois extension with group $G$. 
	\item $E/F$ is an iterative Picard-Vessiot extension with finite differential Galois group $G$. 
\end{enumerate}
\end{thm}

Now we can formulate the main theorem of differential Galois theory in positive characteristic. 

\begin{thm}[\textbf{Galois Correspondence}]
Let $(F,\partial^{*}_{F})$ be an ID-field with algebraically closed field of constants $K$ and let $M \in \emph{\textbf{IDMod}}_F$. Let $E/F$ be an IPV-extension for $M$ and let $\mathcal{G} \le \emph{GL}_{n,K}$ be a reduced linear algebraic group such that $\mathcal{G}(K) = \emph{Gal}_{\emph{ID}}(E/F)$. \\ 
Then 
\begin{enumerate} 
	\item there is an antiisomorphism of the lattices  
		\begin{align*}
			\textit{\textbf{H}} := \{ \mathcal{H} | \; \mathcal{H} \le \mathcal{G} \text{ is a Zariski closed reduced linear algebraic group} \}, 
		\end{align*}
		and 
		\begin{align*}
			\textit{\textbf{L}} := \{ L | \; L \text{ is an ID-field } F \le L \le E, \text{ such that } E/L \text{ is separable} \}. 
		\end{align*} 
		given by $ \Psi: \textit{\textbf{H}} \rightarrow \textit{\textbf{L}}, \mathcal{H} \mapsto E^{\mathcal{H}(K)}$ and $\Psi^{-1}: 				\textit{\textbf{L}} \rightarrow \textit{\textbf{H}}, L \mapsto \mathcal{H}$, where $\mathcal{H}(K) := \emph{Gal}_{\emph{ID}}(E/L)$; 	
	\item if $\mathcal{H} \unlhd \mathcal{G}$ is a Zariski closed reduced normal subgroup, then $L := E^{\mathcal{H}(K)}$ is an IPV-extension of $F$ with Galois group $(\mathcal{G}/\mathcal{H})(K)$.
\end{enumerate}
\end{thm}
\begin{proof}
\cite{Mat01}, Theorem 4.7. 
\end{proof} 

After the basics of the differential Galois theory, we need some definitions about morphisms in $\text{\textbf{AffGr}}_K^{\text{red}}$. 

\begin{defn} \label{def epi}
Let $\beta$$: \widetilde{\mathcal{G}} \rightarrow \mathcal{G}$ be an epimorphism in $\text{\textbf{AffGr}}_K^{\text{red}}$. 
\begin{itemize} 
	\item $\beta$ is called \textbf{split}, if there exists a monomorphism $\sigma$$: \mathcal{G} \hookrightarrow \widetilde{\mathcal{G}}$, such that $\beta \circ \sigma = \text{id}_{\mathcal{G}}$. In this case we say that $\sigma$ is a \textbf{homomorphic section} for $\beta$. 
	\item $\beta$ is called \textbf{direct split}, if $\beta$ is split and $\sigma(\mathcal{G})$ is a normal subgroup of $\widetilde{\mathcal{G}}$  (that means $\widetilde{\mathcal{G}} \cong \text{ker}(\beta) \times \sigma(\mathcal{G})$).
	\item We say that $\beta$ is a \textbf{Frattini-epimorphism}, if $\widetilde{\mathcal{G}}$ is the only reduced closed supplement of $\text{ker}(\beta)$ in $\widetilde{\mathcal{G}}$, i.e., any $\mathcal{U}$ in $\widetilde{\mathcal{G}}$ which is reduced, closed, and satisfies $\text{ker}(\beta)\cdot\mathcal{U} = \widetilde{\mathcal{G}}$ already equals $\widetilde{\mathcal{G}}$. 
\end{itemize}
\end{defn}

The following theorem of A. Borel and J.P. Serre is very important for considering non-connected groups.

\begin{thm} \label{thm BS}
Let $\mathcal{G}$ be a reduced affine group scheme. Then there exists a finite supplement $H$ for the connected component $\mathcal{G}^{\circ}$ in $\mathcal{G}$, i.e., $\mathcal{G} = \mathcal{G}^{\circ} \cdot H$. 
\end{thm}
\begin{proof}
\cite{BS64}, Lemme 5.11. 
\end{proof}

Thus any reduced linear algebraic group $\mathcal{G}$ is a quotient of a semidirect product $\mathcal{G}^{\circ} \rtimes H$. Therefore it is advisable to consider semidirect products. 

\begin{defn}
Let $H$ be a finite group. 
\begin{itemize}
	\item A reduced, connected linear algebraic group $\mathcal{G}^{\circ}$ together with a group homomorphism $H \rightarrow \text{Aut}(\mathcal{G}^{\circ})$, is called an \textbf{$H$-group}. 
	\item Let $\widetilde{\mathcal{G}}^{\circ}$ and $\mathcal{G}^{\circ}$ be $H$-groups. Let $\beta$$: \widetilde{\mathcal{G}}^{\circ} \rightarrow \mathcal{G}^{\circ}$ be a morphism in $\text{\textbf{AffGr}}_K^{\text{red}}$. We call $\beta$ an \textbf{$H$-morphism}, if $\beta$ commutes with the action of $H$. 
	\item We say an $H$-epimorphism is \textbf{$H$-split}, if there exists a homomorphic section which is an $H$-morphism. 
\end{itemize} 
\end{defn}

For any $H$-group $\mathcal{G}^{\circ}$ we can build the semidirect product $\mathcal{G}^{\circ} \rtimes H$. Every $H$-epimorphism $\beta^{\circ}$$: \widetilde{\mathcal{G}}^{\circ} \rightarrow \mathcal{G}^{\circ}$ can be uniquely extended to an epimorphism $\beta$$: \widetilde{\mathcal{G}}^{\circ} \rtimes H \rightarrow \mathcal{G}^{\circ} \rtimes H$, such that $\beta|_H = \text{id}_H$. 

Such epimorphisms are very important for us, so we give them their own name. 

\begin{defn}
Let $H$ be a finite group. Let $\widetilde{\mathcal{G}}^{\circ}$ and $\mathcal{G}^{\circ}$ be two $H$-groups. An epimorphism $\beta$$: \widetilde{\mathcal{G}}^{\circ} \rtimes H \rightarrow \mathcal{G}^{\circ} \rtimes H$ is called \textbf{$H$-rigid}, if $\beta|_H = \text{id}_H$. \\
Then we say $\beta$ is \textbf{$H$-split} if there exists a homomorphic section which is the identity on $H$. \\ 
We call $\beta$ \textbf{subdirect $H$-split}, if $\beta$ is $H$-split and $\widetilde{\mathcal{G}} \cong \text{ker}(\beta) \times \mathcal{G}$. 
\end{defn} 

Let $\beta$$: \widetilde{\mathcal{G}}^{\circ} \rtimes H \rightarrow \mathcal{G}^{\circ} \rtimes H$ be an $H$-rigid epimorphism. Then the restriction $\beta^{\circ} := \beta|_{\widetilde{\mathcal{G}}}$$: \widetilde{\mathcal{G}}^{\circ} \rightarrow \mathcal{G}^{\circ}$ is an $H$-epimorphism, in particular we call $\beta^{\circ}$ the \textbf{connected component} of $\beta$. Note that $\beta$ is $H$-split if and only if $\beta^{\circ}$ is $H$-split.

\section{ID-Embedding Problems}

\begin{defn} \label{def ID-EBP}
Let $F$ be an ID-field with field of constants $K$, let $E/F$ be an IPV-extension with Galois group $\mathcal{G}(K) \cong \text{Gal}_{\text{ID}}(E/F)$, and let $\beta$$: \widetilde{\mathcal{G}} \rightarrow \mathcal{G}$ be an epimorphism in $\text{\textbf{AffGr}}_K^{\text{red}}$. The corresponding \textbf{iterative differential embedding problem} (ID-embedding problem) asks for the existence of an IPV-extension $\widetilde{E}/F$ and a monomorphism $\tilde{\alpha}$ which maps $\text{Gal}_{\text{ID}}(\widetilde{E}/F)$ onto a closed subgroup of $\widetilde{\mathcal{G}}(K)$, such that the diagram  
\begin{align*}
	\begin{xy}
  	\xymatrix{
        &    & \text{Gal}_{\text{ID}}(\widetilde{E}/F) \ar[r]^{\text{res}} \ar[d]_{\widetilde{\alpha}}  & \text{Gal}_{\text{ID}}(E/F) \ar[d]_{\cong}^{\alpha} & \\
     		1 \ar[r] &  \mathcal{A}(K) \ar[r] & \widetilde{\mathcal{G}}(K) \ar[r]^{\beta}  & \mathcal{G}(K) \ar[r] & 1 
  	}
	\end{xy} 
\end{align*} 
commutes. We denote such an ID-embedding problem by $\mathcal{E}(\alpha, \beta)$. The kernel $\mathcal{A}(K)$ of the map $\beta$ is also called the \textbf{kernel} of $\mathcal{E}(\alpha, \beta)$. We say that $\mathcal{E}(\alpha, \beta)$ is \textbf{split}, if $\beta$ is split (respectively direct split, $H$-split, subdirect $H$-split and Frattini). We call $\widetilde{\alpha}$ a \textbf{solution} of $\mathcal{E}(\alpha, \beta)$. Further we say that $\widetilde{\alpha}$ is a \textbf{proper solution} if $\widetilde{\alpha}$ is an isomorphism. 
\end{defn} 

Observe that by definition the affine group schemes $\widetilde{\mathcal{G}}$ and $\mathcal{G}$ are reduced, whereas the kernel $\mathcal{A}$ need not be reduced. \\
In order to solve such an ID-embedding problem we will make a decomposition into ID-embedding problems which are easier to solve. More precisely, a decomposition of an ID-embedding problem means a decomposition of the underlying epimorphism $\beta$ in the following sense: 

\begin{defn} \label{def decom}
Let $\mathcal{M}$ be a finite set of epimorphisms in $\text{\textbf{AffGr}}_K^{\text{red}}$. The set $\tilde{\mathcal{M}}$ is called an \textbf{elementary decomposition} if it is formed in the following way:  
	\begin{enumerate} 
		\item Replace $\beta_2 \in \mathcal{M}$ by epimorphisms $\beta_1$ and $\beta_3$ in $\text{\textbf{AffGr}}_K^{\text{red}}$ with $\beta_2 = \beta_1 \circ \beta_3$. 
		\item Replace $\beta_1 \in \mathcal{M}$ by an epimorphism $\beta_2$ in $\text{\textbf{AffGr}}_K^{\text{red}}$, so that an  epimorphism $\beta_3$ in $\text{\textbf{AffGr}}_K^{\text{red}}$ exists with $\beta_2 = \beta_1 \circ \beta_3$. 
	\end{enumerate} 
Let $\beta$$: \widetilde{\mathcal{G}} \rightarrow \mathcal{G}$ be an epimorphism in $\text{\textbf{AffGr}}_K^{\text{red}}$. Then $\beta$ is called  \textbf{decomposable} into $\beta_1, \ldots, \beta_n$, if there exist sets $\mathcal{M}_i$ ($1 \le i \le m$), with $\mathcal{M}_1 = \{\beta\}$ and $\mathcal{M}_m = \{\beta_1, \ldots, \beta_n\}$, such that $\mathcal{M}_{i+1 }$ is an elementary decomposition of $\mathcal{M}_i$ for all $i$. 
\end{defn} 

The next proposition shows how a decomposition of the underlying epimorphism affects the ID-embedding problem: 

\begin{prop} \label{prop decom ebp}
Let the following commutative diagram of epimorphisms in $\emph{\textbf{AffGr}}_K^{\emph{red}}$ be given. 
\begin{align*}
	\begin{xy}
  	\xymatrix{
      	&	\mathcal{G}_2	\ar@{->>}[rd]^{\beta_1} &	\\
      	\mathcal{G}_1 \ar@{->>}[rr]^{\beta_2} \ar@{->>}[ru]^{\beta_3} &	& \mathcal{G}_3 
  	}
	\end{xy} 
\end{align*}
\begin{enumerate}
	\item Let $\widetilde{\alpha}$$: \emph{Gal}_{\emph{ID}}(\widetilde{E}/F) \rightarrow \mathcal{G}_2(K)$ be a proper solution of the ID-embedding problem $\mathcal{E}(\alpha, \beta_1)$ and let $\overline{\alpha}$$: \emph{Gal}_{\emph{ID}}(\overline{E}/F) \rightarrow \mathcal{G}_1(K)$ be a (proper) solution of the ID-embedding problem $\mathcal{E}(\widetilde{\alpha}, \beta_3)$. \\
Then $\overline{\alpha}$ is also a (proper) solution of the ID-embedding problem $\mathcal{E}(\alpha, \beta_2)$.
	\item If conversely $\overline{\alpha}$$: \emph{Gal}_{\emph{ID}}(\overline{E}/F) \rightarrow \mathcal{G}_1(K)$ is a proper solution of the ID-embedding problem $\mathcal{E}(\alpha, \beta_2)$ and 
		\begin{align*} 
			\widetilde{E} := \overline{E}^{\emph{ker}(\beta_3)}, 
		\end{align*}
		then there exists exactly one monomorphism $\widetilde{\alpha}$$: \emph{Gal}_{\emph{ID}}(\widetilde{E}/F) \rightarrow \mathcal{G}_2(K)$, such that 
		\begin{align*} \tag{$\ast$}
			\widetilde{\alpha} \circ \emph{res} = \beta_3 \circ \overline{\alpha}.  
		\end{align*} 
		In addition $\tilde{\alpha}$ is a proper solution of the ID-embedding problem $\mathcal{E}(\alpha, \beta_1)$. 
\end{enumerate}
\end{prop}
\begin{proof}
\begin{enumerate}
	\item By the assumptions we have the following commutative diagram, which proves the claim: 
		\begin{align*}
		\begin{xy}
  	\xymatrix{
  			& \text{Gal}_{\text{ID}}(\widetilde{E}/F) \ar[d]^{\widetilde{\alpha}} \ar[rd]^{\text{res}} & \\
      	\text{Gal}_{\text{ID}}(\overline{E}/F) \ar[d]^{\overline{\alpha}} \ar[ru]^{\text{res}} &	\mathcal{G}_2(K)	\ar@{->>}[rd]^{\beta_1} &	\text{Gal}_{\text{ID}}(E/F) \ar[d]^{\alpha} \\
      	\mathcal{G}_1(K) \ar@{->>}[rr]^{\beta_2} \ar@{->>}[ru]^{\beta_3} &	& \mathcal{G}_3(K)  
  	}
		\end{xy} 
		\end{align*}
	\item We define $\widetilde{\alpha}$ by equation $(\ast)$. Then $\widetilde{\alpha}$ is well-defined, because $\beta_3$ and the restriction map have the same kernel. Since the restriction map is surjective, $\widetilde{\alpha}$	is unique. It is obvious that $\tilde{\alpha}$ is a proper solution of $\mathcal{E}(\alpha, \beta_1)$. 
\end{enumerate}
\end{proof}
By the second part of this proposition we can reduce the ID-embedding problem $\mathcal{E}(\alpha, \beta_1)$ to the ID-embedding problem $\mathcal{E}(\alpha, \beta_2)$. The first part provides a decomposition of the ID-embedding problem $\mathcal{E}(\alpha, \beta_2)$ into the ID-embedding problems $\mathcal{E}(\alpha, \beta_1)$ and $\mathcal{E}(\widetilde{\alpha}, \beta_3)$. But in this case the ID-embedding problem $\mathcal{E}(\widetilde{\alpha}, \beta_3)$ depends on the selected proper solution from $\mathcal{E}(\alpha, \beta_1)$. This problem can be attacked by the following definition.

\begin{defn} \label{defn emb}
An epimorphism in $\text{\textbf{AffGr}}_K^{\text{red}}$ is called an \textbf{embedding epimorphism} (over $F$) if all corresponding ID-embedding problems over $F$ have a proper solution. 
\end{defn}

The next proposition is a consequence of Proposition \ref{prop decom ebp}.

\begin{prop} \label{prop emb}
Let $\beta$$: \widetilde{\mathcal{G}} \rightarrow \mathcal{G}$ be an epimorphism in $\emph{\textbf{AffGr}}_K^{\emph{red}}$. If $\beta$ is decomposable into the embedding epimorphisms $\beta_1, \ldots, \beta_n$, then $\beta$ is an embedding epimorphism.  
\end{prop}

\section{Decomposition of ID-Embedding Problems} 

All results of this section are taken from \cite{Obe03}. We give the proofs only for completeness, since the work of T. Oberlies \cite{Obe03} is up to now only published in german. \\ 
The first lemma gives us the backround to use the theory of algebraic groups to decompose an epimorphism. More precisely, a ``decomposition'' of the kernel induces a decomposition of the epimorphism. 

\begin{lemma} \label{lem decom1}
Let $\beta$$: \widetilde{\mathcal{G}} \rightarrow \mathcal{G}$ be an epimorphism in $\emph{\textbf{AffGr}}_K^{\emph{red}}$ with kernel $\mathcal{A}$. Let $\mathcal{V}$ be a closed subgroup of $\mathcal{A}$ which is normal in $\widetilde{\mathcal{G}}$. Then $\beta$ is decomposable into the epimorphism $\beta_1$$: \widetilde{\mathcal{G}} \rightarrow \widetilde{\mathcal{G}}/\mathcal{V}$ and the epimorphism $\beta_2$$: \widetilde{\mathcal{G}}/\mathcal{V} \rightarrow \mathcal{G}$. \\ 
If in addition $\beta$ is split, then $\beta_2$ is also split. 
\end{lemma} 
\begin{proof}
See Definition \ref{def decom}. 
\end{proof}

The next lemma show us that an epimorphism can be decomposed into epimorphisms of extreme cases.

\begin{lemma} \label{lem decom2}
Let $\beta$$: \widetilde{\mathcal{G}} \rightarrow \mathcal{G}$ be an epimorphism in $\emph{\textbf{AffGr}}_K^{\emph{red}}$. Then $\beta$ is decomposable into a Frattini-epimorphism and a split epimorphism (with the same kernel as $\beta$). 
\end{lemma}
\begin{proof}
The reduced subgroups of $\widetilde{\mathcal{G}}$ correspond to the radical ideals of $K[\widetilde{\mathcal{G}}]$, the coordinate ring of $\widetilde{\mathcal{G}}$. Since $K[\widetilde{\mathcal{G}}]$ is Noetherian, there exists a minimal, reduced, closed supplement $\mathcal{U}$ for $\mathcal{A}$ in $\widetilde{\mathcal{G}}$. Therefore $\widetilde{\mathcal{G}} = \mathcal{U} \cdot \mathcal{A}$ and the following commutative diagram arises 
\begin{align*}
\begin{xy}
  \xymatrix{
      1 \ar[r] &  \mathcal{A} \ar[r] \ar[d]_{\cong} & \mathcal{A} \rtimes \mathcal{U} \ar[r]_{\text{pr}_{\mathcal{U}}} \ar@{->>}[d]^{\psi} & \mathcal{U} \ar@<-4pt>@{.>}[l] \ar[r] \ar@{->>}[d]^{\beta|_{\mathcal{U}}} & 1 \\
      1 \ar[r] &  \mathcal{A} \ar[r] & \widetilde{\mathcal{G}} \ar[r]^{\beta} & \mathcal{G} \ar[r] & 1, 
  }
\end{xy} 
\end{align*}
where $\psi$ maps $(a,u)$ to the product $a \cdot u$ in $\widetilde{\mathcal{G}}$. By Definition \ref{def decom} we see that $\beta$ is decomposable into $\beta|_{\mathcal{U}} \circ \text{pr}_{\mathcal{U}}$, which is decomposable into $\beta|_{\mathcal{U}}$ and $\text{pr}_{\mathcal{U}}$. Since $\mathcal{U}$ is minimal it follows that $\beta|_{\mathcal{U}}$ is a Frattini-epimorphism. 
\end{proof} 

Frattini ID-embedding problems and split ID-embedding problems are extreme cases in the following sense: Every split ID-embedding problem is solvable by definition. But it is questionable whether the solution is proper. For a Frattini ID-embedding problem the existence of a solution is not obvious. But if a solution exists, it is obvious by definition that this solution is proper (see Proposition \ref{prop frat ep}). \\
For our further proceeding we need the following classification of Frattini- epimorphisms with finite kernel. 

\begin{lemma} \label{lem frat class}
Let $\beta$$: \widetilde{\mathcal{G}} \rightarrow \mathcal{G}$ be an epimorphism in $\emph{\textbf{AffGr}}_K^{\emph{red}}$ with finite kernel $\mathcal{A}$. Let $\kappa$$: \widetilde{\mathcal{G}} \twoheadrightarrow \widetilde{\mathcal{G}}/\widetilde{\mathcal{G}}^{\circ}$ be the canonical epimorphism and let $\Phi(\widetilde{\mathcal{G}}/\widetilde{\mathcal{G}}^{\circ})$ the Frattini-subgroup of the finite group $\widetilde{\mathcal{G}}/\widetilde{\mathcal{G}}^{\circ}$. Then the following statements are equivalent: 
\begin{enumerate}
	\item $\beta$ is a Frattini-epimorphism 
	\item $\kappa(\mathcal{A}) \subseteq \Phi(\widetilde{\mathcal{G}}/\widetilde{\mathcal{G}}^{\circ})$. 
\end{enumerate}
\end{lemma}
\begin{proof} 
'$\Longleftarrow$': Assume $\beta$ is not a Frattini-epimorphism. There exists a subgroup $\mathcal{B} \subsetneqq \widetilde{\mathcal{G}}$, such that $\widetilde{\mathcal{G}} = \mathcal{A} \cdot \mathcal{B}$. In this case $\text{ker}(\kappa) = \widetilde{\mathcal{G}}^{\circ} = \mathcal{B}^{\circ} \subseteq \mathcal{B}$ by dimension arguments. If $\kappa(\mathcal{B}) = \widetilde{\mathcal{G}}/\widetilde{\mathcal{G}}^{\circ}$, it follows that $\widetilde{\mathcal{G}} = \mathcal{B}$ (contradiction). Therefore $\kappa(\mathcal{B}) \subsetneqq \widetilde{\mathcal{G}}/\widetilde{\mathcal{G}}^{\circ}$ holds. Thus $\kappa(\mathcal{A}) \cdot \kappa(\widetilde{\mathcal{B}}) = \widetilde{\mathcal{G}}/\widetilde{\mathcal{G}}^{\circ}$ together with \cite{Hup67}, Satz III.2 implies that $\kappa(\mathcal{A}) \nsubseteq \Phi(\widetilde{\mathcal{G}}/\widetilde{\mathcal{G}}^{\circ})$. \\ 
'$\Longrightarrow$': Assume $\kappa(\mathcal{A}) \nsubseteq \Phi(\widetilde{\mathcal{G}}/\widetilde{\mathcal{G}}^{\circ})$. By \cite{Hup67}, Satz III.2 there exists a subgroup $\mathcal{S} \subsetneqq \widetilde{\mathcal{G}}/\widetilde{\mathcal{G}}^{\circ}$ with $\kappa(\mathcal{A}) \cdot \mathcal{S} = \widetilde{\mathcal{G}}/\widetilde{\mathcal{G}}^{\circ}$. With $\mathcal{B} := \kappa^{-1}(\mathcal{S})$ we have $\text{ker}(\kappa) \subseteq \mathcal{B} \subseteq \mathcal{B} \cdot \mathcal{A}$ and $\kappa(\mathcal{B} \cdot \mathcal{A}) = \widetilde{\mathcal{G}}/\widetilde{\mathcal{G}}^{\circ}$. Therefore $\mathcal{B}$$\cdot$$\mathcal{A} = \widetilde{\mathcal{G}}$ and we have that $\beta$ is not a Frattini-epimorphism, since $\mathcal{B} \subsetneqq \widetilde{\mathcal{G}}$.
\end{proof} 

The next lemmata are needed for the decomposition into $H$-rigid epimorphisms. 

\begin{lemma} \label{lem H-fratt}
Let $\mathcal{G} \in \emph{\textbf{AffGr}}_K^{\emph{red}}$. Then there exists a finite subgroup $H$ of $\mathcal{G}$ with $\mathcal{G} = \mathcal{G}^{\circ} \cdot H$, such that $\mu$$: \mathcal{G}^{\circ} \rtimes H \rightarrow \mathcal{G}$, $(g,h) \mapsto g \cdot h$ is a Frattini-epimorphism. 
\end{lemma}
\begin{proof}
By Theorem \ref{thm BS} there exists a finite subgroup $H$ of $\mathcal{G}$ with $\mathcal{G} = \mathcal{G}^{\circ} \cdot H$. Assume that $H$ is minimal with this property. \\ 
Then $H \rightarrow H/H \cap \mathcal{G}^{\circ}$ is a Frattini-epimorphism. Since, if $W$ is a proper subgroup of $H$ with $W \cdot (H \cap \mathcal{G}^{\circ}) = H$, then $\mathcal{G} = W \cdot \mathcal{G}^{\circ}$, which is a contradiction to the minimality of $H$. \\ 
By applying Lemma \ref{lem frat class} to the epimorphism $H \rightarrow H/H \cap \mathcal{G}^{\circ}$, we see that $H \cap \mathcal{G}^{\circ}$ lies in the Frattini-subgroup $\Phi(H)$ of $H$. \\ 
Next we will apply Lemma \ref{lem frat class} to $\mu$. In this case the canonical epimorphism $\kappa$ is the projection $\mathcal{G}^{\circ} \rtimes H \rightarrow H$. So we have
\begin{align*}
	\text{ker}(\mu) = \{ (g^{-1},g)| g \in H \cap \mathcal{G}^{\circ} \}
\end{align*} 
and therefore $\kappa(\text{ker}(\mu)) = H \cap \mathcal{G}^{\circ} \le \Phi(H)$ holds. 
\end{proof}

\begin{defn} \label{defn mu}
We call such an Frattini-epimorphism $\mu$$: \mathcal{G}^{\circ} \rtimes H \rightarrow \mathcal{G}^{\circ} \cdot H$, $(g,h) \mapsto g \cdot h$, an epimorphism \textbf{of type $\mu$}. 
\end{defn}

\begin{lemma} \label{lem H-semi split decom}
Let $\beta$$: \mathcal{A} \rtimes \mathcal{G}  \rightarrow \mathcal{G}$ be an epimorphism in $\emph{\textbf{AffGr}}_K^{\emph{red}}$ with reduced, connected kernel $\mathcal{A}$. Then there exists a finite subgroup $H$ of $\mathcal{G}$, such that $\beta$ is decomposable into a Frattini-epimorphism $\mu$ with finite kernel and an $H$-split, $H$-rigid epimorphism $\overline{\beta}$ with kernel $\mathcal{A}$. 
\end{lemma}
\begin{proof}
Let $H$ be a finite subgroup of $\mathcal{G}$ with the properties of Lemma \ref{lem H-fratt}. Therefore $\mu$$: \mathcal{G}^{\circ} \rtimes H \rightarrow \mathcal{G}$, $(g,h) \mapsto g \cdot h$ is a Frattini-epimorphism. The epimorphism $\beta|_{\mathcal{A} \rtimes \mathcal{G}^{\circ}}$$:~\mathcal{A} \rtimes \mathcal{G}^{\circ} \rightarrow \mathcal{G}^{\circ}$ is an $H$-epimorphism, since $H$ acts via conjugation on $\mathcal{A}$ and $\mathcal{G}^{\circ}$. Therefore, we consider the corresponding $H$-rigid, $H$-split epimorphism 
\begin{align*}
	\overline{\beta}: (\mathcal{A} \rtimes \mathcal{G}^{\circ}) \rtimes H \rightarrow \mathcal{G}^{\circ} \rtimes \mathcal{G}. 
\end{align*} 
By defining the epimorphism 
\begin{align*}
	(\text{id},\mu): (\mathcal{A} \rtimes \mathcal{G}^{\circ}) \rtimes H & \rightarrow \mathcal{A} \rtimes H, \\
	 ((a,g),h) & \mapsto (a,g \cdot h), 
\end{align*} 
we obtain the following commutative diagram
\begin{align*}
	\begin{xy}
  	\xymatrix{
        \mathcal{A} \rtimes \mathcal{G} \ar[r]^{\beta} &  \mathcal{G} \\
     		(\mathcal{A} \rtimes \mathcal{G}^{\circ}) \rtimes H \ar[u]^{(\text{id},\mu)} \ar[r]^/.3cm/{\overline{\beta}}  & \mathcal{G}^{\circ} \rtimes H \ar[u]_{\mu}. 
  	}
	\end{xy} 
\end{align*} 
Which proves the claim. 
\end{proof}

\begin{lemma} \label{lem H-semi frat decom}
Let $\beta$$: \widetilde{\mathcal{G}}  \rightarrow \mathcal{G}$ be a Frattini-epimorphism in $\emph{\textbf{AffGr}}_K^{\emph{red}}$ with reduced kernel $\mathcal{A}$. Then there exists a finite subgroup $H$ of $\widetilde{\mathcal{G}}$, such that $\beta$ is decomposable into a Frattini-epimorphism $\mu$ with finite kernel and an $H$-rigid Frattini-epimorphism $\overline{\beta}$ with kernel $\mathcal{A}$. 
\end{lemma} 
\begin{proof}
Let $H$ be a finite subgroup of $\widetilde{\mathcal{G}}$ with the properties of Lemma \ref{lem H-fratt}. Therefore $\widetilde{\mu}$$: \widetilde{\mathcal{G}}^{\circ} \rtimes H \rightarrow \widetilde{\mathcal{G}}$, $(g,h) \mapsto g \cdot h$ is a Frattini-epimorphism. Since $\beta(\widetilde{\mathcal{G}}^{\circ}) = \mathcal{G}^{\circ}$, the group $\mathcal{G}^{\circ}$ is an $H$-group via $h \cdot \beta(g) := \beta(hgh^{-1})$. So the restriction $\beta|_{\widetilde{\mathcal{G}}^{\circ}}$ is an $H$-epimorphism and induces the $H$-rigid epimorphism $\overline{\beta}$$: \widetilde{\mathcal{G}}^{\circ} \rtimes H  \rightarrow \mathcal{G}^{\circ} \rtimes H$. By defining the epimorphism 
\begin{align*}
	\mu: \mathcal{G}^{\circ} \rtimes H \rightarrow \mathcal{G}, (g,h) \mapsto g \cdot \beta(h), 
\end{align*} 
we obtain the following commutative diagram 
\begin{align*}
	\begin{xy}
  	\xymatrix{
        \mathcal{A}  \ar@{^{(}->}[r] &  \widetilde{\mathcal{G}} \ar[r]^{\beta} &  \mathcal{G} \\ 
     		\mathcal{A}  \ar@{^{(}->}[r] &  \widetilde{\mathcal{G}}^{\circ} \rtimes H \ar[u]^{\widetilde{\mu}} \ar[r]^{\overline{\beta}}  & \mathcal{G}^{\circ} \rtimes H \ar[u]_{\mu}. 
  	}
	\end{xy} 
\end{align*} 
Which proves the claim. 
\end{proof} 

Now we have all parts together to prove a helpful decomposition of epimorphisms. 

\begin{prop} \label{prop decompostion} 
Let $\beta$$: \widetilde{\mathcal{G}} \rightarrow \mathcal{G}$ be an epimorphism in $\emph{\textbf{AffGr}}_K^{\emph{red}}$ with reduced kernel $\mathcal{A}$. Then $\beta$ is decomposable into 
\begin{enumerate}
	\item epimorphisms with finite kernel, 
	\item Frattini-epimorphisms, which are $H$-rigid with respect to a finite group $H$, 
	\item epimorphisms with reduced, connected, semi-simple, centerless kernel, which are $H$-rigid, $H$-split with respect to a finite group $H$, 
	\item epimorphisms with reduced, torus kernel, which are $H$-rigid, $H$-split with respect to a finite group $H$, 
	\item epimorphisms with minimal, reduced, connected, unipotent kernel, which are $H$-rigid, $H$-split with respect to a finite group $H$. 
\end{enumerate}
\end{prop} 
\begin{proof} 
We prove the proposition in nine steps: \\ 
\emph{1. The epimorphism $\beta$ is decomposable into an epimorphism with reduced, connected kernel and an epimorphism with finite kernel:} \\
Since $\mathcal{A}^{\circ}$ is a reduced, characteristic subgroup of $\mathcal{A}$ it is normal in $\widetilde{\mathcal{G}}$. With Lemma \ref{lem decom1} the claim follows. \\ 
\emph{2. Let $\beta$ be an epimorphism with reduced, connected kernel. Then $\beta$ is decomposable into an epimorphism with reduced, connected, semi-simple kernel and an epimorphism with reduced, connected, solvable kernel:} \\
The radical $R(\mathcal{A})$ is a reduced, connected, characteristic subgroup of $\mathcal{A}$ and hence normal in $\widetilde{\mathcal{G}}$. By Lemma \ref{lem decom1} the claim follows. \\
\emph{3. Let $\beta$ be an epimorphism with reduced, connected, semi-simple kernel $\mathcal{A}$. Then $\beta$ is decomposable into an epimorphism with reduced, abelian kernel and an epimorphism with reduced, connected, semi-simple, centerless kernel:} \\
The center $\emph{Z}(\mathcal{A})$ is a reduced, characteristic subgroup of $\mathcal{A}$ and hence normal in $\widetilde{\mathcal{G}}$. Again by Lemma \ref{lem decom1} the claim follows. \\ 
\emph{4. Let $\beta$ be an epimorphism with reduced, connected, solvable kernel. Then $\beta$ is decomposable into epimorphisms with reduced, connected, abelian kernel:} \\
If $\mathcal{A}$ is abelian, there is nothing to do. Let $A$ not be abelian, then the dimension of the commutator group $[\mathcal{A},\mathcal{A}]$ is strictly smaller than the dimension of $\mathcal{A}$. Since $[\mathcal{A},\mathcal{A}]$ is a reduced, characteristic subgroup of $\mathcal{A}$ it is normal in $\widetilde{\mathcal{G}}$. With Lemma \ref{lem decom1}, $\beta$ is decomposable into a reduced epimorphism with reduced, connected, abelian kernel and an epimorphism with reduced, connected, solvable kernel $[\mathcal{A},\mathcal{A}]$. Iteration of this process yields the result. \\
\emph{5. Let $\beta$ be an epimorphism with reduced, connected, abelian kernel. Then $\beta$ is decomposable into an epimorphism with reduced, connected, unipotent kernel and an epimorphism with reduced torus kernel:} \\ 
The set $\mathcal{A}_u$ of the unipotent elements of $\mathcal{A}$ is a reduced, connected, characteristic subgroup of $\mathcal{A}$ and hence normal in $\widetilde{\mathcal{G}}$. Since $\mathcal{A}/\mathcal{A}_u$ is a reduced torus the claim follows by Lemma \ref{lem decom1}. \\ 
\emph{6. Let $\beta$ be an epimorphism with reduced, connected, unipotent kernel. Then $\beta$ is decomposable into epimorphisms with minimal, reduced, connected, unipotent kernel:} \\ 
Use induction on the dimension of the kernel together with Lemma \ref{lem decom1}. \\
\emph{7. Let $\beta$ be an epimorphism with reduced, connected kernel. Then $\beta$ is decomposable into a Frattini-epimorphism and a split epimorphism with the same kernel like $\beta$:} \\ 
See Lemma \ref{lem decom2}. \\
\emph{8. Let $\beta$ be a split epimorphism with reduced, connected kernel. Then $\beta$ is decomposable into an epimorphism with finite kernel and an $H$-rigid, $H$-split epimorphism with the same kernel like $\beta$:} \\ 
See Lemma \ref{lem H-semi split decom}. \\ 
\emph{9. Let $\beta$ be a Frattini-epimorphism with reduced kernel. Then $\beta$ is decomposable into an epimorphism with finite kernel and an $H$-rigid, Frattini-epimorphism with the same kernel like $\beta$:} \\ 
See Lemma \ref{lem H-semi frat decom}. \\
\end{proof}

\section{H-Effective Embedding Problems}

Before we start to solve the ID-embedding problems, we take a look at the Galois theory of the semidirect product $\mathcal{G}^{\circ} \rtimes H$, where $H$ is a finite group.

\begin{defn}
Let $F$ be an ID-field and $E/F$ be an IPV-extension, which is defined by matrices $D_l \in \text{GL}_{n}(F_{l})$ and Galois group $\mathcal{G}(K) \cong \text{Gal}_{\text{ID}}(E/F)$ (the matrices $D_l$ were defined in section 2). Then we call $E/F$ \textbf{effective}, if $D_l \in \mathcal{G}(F_{l})$ for all $l \in \mathbb{N}$. 
\end{defn}

\begin{remark}
If $E/F$ is effective, then one can show that the Galois group $\text{Gal}_{\text{ID}}(E/F)$ is necessarily connected. 
\end{remark}

Now we show that for an important class of fields the converse holds.

\begin{remark} 
A field $F$ has cohomological dimension $\le 1$ ($\text{cd}(F) \le 1$) if $F$ is the only central division algebra over $F$. For more information about the cohomological dimension see \cite{Ser97}. Important examples of fields of cohomological dimension $\le 1$ are $K(t)$, $K((t))$ and algebraic extensions, where $K$ is an algebraically closed field (see \cite{Ser97}, II.3.3). 
\end{remark}

\begin{thm} \label{thm cd}
Let $F$ be an ID-field with $\emph{cd}(F) \le$ \emph{1} and with algebraically closed field of constants $K$. Let $\mathcal{H} \le \emph{GL}_{n,K}$ be a reduced, connected linear algebraic group. Let $R/F$ be an IPV-ring and assume that the Galois group $\mathcal{G}(K) \cong \emph{Gal}_{\emph{ID}}(R/F)$ is connected. Suppose that the defining matrices $D_l$ satisfy $D_l \in \mathcal{H}(F_l)$. Then there exist matrices $C_l \in \mathcal{H}(F_l)$ such that $\widetilde{D}_l := C_{l}D_{l}C_{l+1}^{-1} \in \mathcal{G}(F_l)$.
\end{thm}
\begin{proof}
\cite{Mat01}, Theorem 5.9.
\end{proof}

\begin{remark} \label{rem gal iso}
Let $L/F$ be a finite IPV-extension. Let $H$ be a finite group and let $\overline{\alpha}$$: \text{Gal}(L/F) \rightarrow H$ be an isomorphism of groups. Since any element in $\text{Gal}(L/F)$ is an ID-automorphism, the map $\text{Gal}(L/F) \stackrel{\text{res}}{\rightarrow} \text{Gal}(L_l/F_l)$ is an isomorphism of groups. Then there exists exactly one isomorphism $\overline{\alpha}_l$$: \text{Gal}(L_l/F_l) \rightarrow H$, such that the following diagram commutes: 
\begin{align*}
	\begin{xy}
  	\xymatrix{
        \text{Gal}(L/F) \ar[r]^/0.3cm/{\overline{\alpha}} \ar[d]_{\text{res}}^{\cong}  & H \\
     		\text{Gal}(L_l/F_l) \ar[ru]_{\overline{\alpha}_l} 
  	}
	\end{xy} 
\end{align*} 
\end{remark}

\begin{prop} \label{prop semidirect}
Let $F$ be an ID-field with field of constants $K$ and $\emph{cd}(F) \le \emph{1}$. Let $\mathcal{G} := \mathcal{G}^{\circ} \rtimes H \le \emph{GL}_{n,K}$, with regular homomorphic section $\sigma$$: H \rightarrow \mathcal{G}(K)$. Further let $E/F$ be an IPV-extension, with $\alpha$$: \emph{Gal}_{\emph{ID}}(E/F) \stackrel{\cong}{\rightarrow} \mathcal{G}(K)$ and fundamental solution matrices $\widetilde{Y}_l \in \emph{GL}_{n}(E_l)$. 
\begin{enumerate}
	\item Let $L := E^{\mathcal{G}^{\circ}(K)}$. Then $L/F$ is a finite IPV-extension with Galois group isomorphic to $H$ via $\overline{\alpha}$ and for all $l \in \mathbb{N}$ there exist elements $Z_l \in \emph{GL}_n(L_l)$ satisfying $\eta(Z_l) = Z_lC_{\eta}$ for all $\eta \in H \cong \emph{Gal}_{\emph{ID}}(L/F)$, where $C_{\eta} = (\sigma \circ \overline{\alpha})(\eta)$. 
	\item $E/L$ is an effective IPV-extension with Galois group isomorphic to $\mathcal{G}^{\circ}(K)$ and fundamental solution matices $Y_l := Z_l^{-1}\widetilde{Y}_l$, such that $\epsilon(Y_l) = Y_lC_{\epsilon}$ for all $\epsilon \in \emph{Gal}_{\emph{ID}}(E/L)$ and $\widetilde{\eta}(Y_l) = C^{-1}_{\eta}Y_lC_{\eta}$ for all $\eta \in \emph{Gal}_{\emph{ID}}(L/F)$ (where $C_{\epsilon} = \alpha(\epsilon)$ and $\widetilde{\eta} := (\alpha^{-1} \circ \sigma \circ \overline{\alpha})(\eta)$). 
\end{enumerate}
\end{prop}
\begin{proof}
 \emph{(1)} Since $\widetilde{Y}_l \in \text{GL}_{n}(E_l)$ are fundamental solution matrices, we have that $\gamma(Y_l) =Y_lC_{\gamma}$ for all $\gamma \in \text{Gal}_{\text{ID}}(E/F)$. By Hilbert's Theorem 90 (\cite{Ser97}, III.1, Lemma 1), there exists an element $Z_0 \in \text{GL}_n(L)$ with $\eta(Z_0) = Z_0C_{\eta}$ for all $\eta \in \text{Gal}_{\text{ID}}(L/F)$. For any $\eta \in \text{Gal}_{\text{ID}}(L/F(Z_0))$ the equation $(\sigma \circ \overline{\alpha})(\eta) = C_{\eta} = Z_0^{-1}\eta(Z_0) = 1$
holds and so $\eta = \text{id}_L$, hence $L = F(Z_0)$. Again by Hilbert's Theorem 90, there exists an element $Z_1 \in \text{GL}_n(L_1)$ such that $\eta(Z_1) = Z_1C_{\eta}$ for all $\eta \in \text{Gal}_{\text{ID}}(L_1/F_1) \cong \text{Gal}_{\text{ID}}(L/F)$. Iteration of this process yields the result. \\
 \emph{(2)} By defining $Y_l := Z^{-1}_l\widetilde{Y}_l$, we get that $E = L(Y_0)$ and $E/L$ is effective, because $\text{cd}(L) \le \text{1}$ (see Theorem \ref{thm cd}). Let $\widetilde{D}_l$ be matrices which define the IPV-extension $E/F$, i.e. $\widetilde{Y}_{l+1} = \widetilde{D}^{-1}_l\widetilde{Y}_l$. Therefore $Y_{l+1} = Z^{-1}_{l+1}\widetilde{Y}_{l+1} = Z^{-1}_{l+1}\widetilde{D}^{-1}_l\widetilde{Y}_l = Z^{-1}_{l+1}\widetilde{D}^{-1}_lZ_lY_l$ holds and $D_l := Z^{-1}_{l+1}\widetilde{D}^{-1}_lZ_l$ are matrices associated to the fundamental solution matrices $Y_l$. Further we have $\epsilon(Y_l) = Z^{-1}_l\epsilon(\widetilde{Y}_l) = Z^{-1}\widetilde{Y}_lC_{\epsilon} = Y_lC_{\epsilon}$ for all $\epsilon \in \text{Gal}_{\text{ID}}(E/L)$ and $\widetilde{\eta}(Y_l) = \eta(Z^{-1}_l)\widetilde{\eta}(\widetilde{Y}_l) = C^{-1}_{\eta}Z^{-1}_l\widetilde{Y}_lC_{\eta} = C^{-1}_{\eta}Y_lC_{\eta}$ for all $\eta \in \text{Gal}_{\text{ID}}(L/F)$. 
\end{proof}

\begin{defn}
Let $H$ be a reduced linear algebraic group. Let $L/F$ be an IPV-extension and let $\overline{\alpha}$$: \text{Gal}_{\text{ID}}(L/F) \rightarrow H(K)$ be an isomorphism of groups. Let $\mathcal{G} := \mathcal{G}^{\circ} \rtimes H \le \text{GL}_{n,K}$, with regular homomorphic section $\sigma$$: H(K) \rightarrow \mathcal{G}(K)$. We call matrices $D_l \in \mathcal{G}^{\circ}(L_l)$ \textbf{$H$-equivariant} (via $\alpha$) if for all $\eta \in H$ 
\begin{align*}
	\eta(D_l) = \chi(\eta)^{-1}D_l\chi(\eta)
\end{align*}
holds for all $l \in \mathbb{N}$. Here the action of $\eta$ on the left-hand side is the (coefficient-wise) Galois action, while on the right-hand side $\eta$ via the monomorphism
\begin{align*}
	\chi := \sigma \circ \overline{\alpha}: \text{Gal}_{\text{ID}}(L/F) \rightarrow H(K) \rightarrow (\mathcal{G}^{\circ} \rtimes H)(K), 
\end{align*} 
constructed as an element of $(\mathcal{G} \rtimes H)(K)$, which acts via conjugation on $\mathcal{G}(F_l)$. 
\end{defn}

\begin{thm} \label{thm semidirect}
Let $H$ be a finite group and let $\mathcal{G} := \mathcal{G}^{\circ} \rtimes H \le \emph{GL}_{n,K}$, with regular homomorphic section $\sigma$$: H \rightarrow \mathcal{G}(K)$. Further, let $F$ be an ID-field with field of constants $K$ and $\emph{cd}(F) \le \emph{1}$. 
\begin{enumerate}
	\item Let $L/F$ be a finite IPV-extension with Galois group isomorphic to $H$ via $\overline{\alpha}$. Let 
	\begin{align*}
		\chi := \sigma \circ \overline{\alpha}: \emph{Gal}_{\emph{ID}}(L/F) \rightarrow \sigma(H) \le \emph{GL}_{n}(K), \ \ \ \eta \mapsto C_{\eta}
	\end{align*}
	be the composite isomorphism. Then for all $l \in \mathbb{N}$ there exist elements $Z_l \in \emph{GL}_n(L_l)$ satisfying $\eta(Z_l) = Z_lC_{\eta}$ for all $\eta \in H \cong \emph{Gal}_{\emph{ID}}(L/F)$ and $C_l \in \emph{GL}_n(F_l)$ such that $Z_{l+1} = C^{-1}_lZ_l$. Moreover, $L = F(Z)$ with $Z := Z_0$. In other words, $Z$ is a fundamental solution matrix for the extension $L/F$ on which the Galois group $\emph{Gal}_{\emph{ID}}(L/F)$ acts via $\overline{\alpha}$. 
	\item Let $E/L$ be an IPV-extension with Galois group isomorphic to $\mathcal{G}^{\circ}(K)$ via an isomorphism 
	\begin{align*}
		\alpha_L : \emph{Gal}_{\emph{ID}}(E/L) \rightarrow \mathcal{G}^{\circ}(K) \unlhd \mathcal{G}(K), \ \ \ \epsilon \mapsto C_{\epsilon}. 
	\end{align*} 
	Then there exist elements $Y_l \in \mathcal{G}(E_l)$ satisfying $\epsilon(Y_l) = Y_lC_{\epsilon}$ for all $\epsilon \in \emph{Gal}_{\emph{ID}}(E/L)$ and $D_l \in \mathcal{G}^{\circ}(L_l)$ such that $Y_{l+1} = D^{-1}_lY_l$. Moreover $E = L(Y)$, with $Y := Y_0$. In other words, $Y$ is a fundamental solution matrix for the extension $E/L$ on which the Galois group $\emph{Gal}_{\emph{ID}}(E/L)$ acts via $\alpha_L$. 
	\item Suppose in addition that $D_l$ are $H$-equivariant via $\overline{\alpha}$: 
	\begin{align*}
	\eta(D_l) = C^{-1}_{\eta}D_lC_{\eta} \ \ \ \text{for all} \ \ l \in \mathbb{N}, \eta \in H.
	\end{align*}
	Then $E/F$ is an IPV-extension with Galois group isomorphic to $\mathcal{G}(K)$ and $\widetilde{Y} := ZY$ is a fundamental solution matrix of this extension which satisfies $\widetilde{Y}_{l+1} = \widetilde{D}^{-1}_l\widetilde{Y}_l \in \emph{GL}_n(E_l)$ for all $l \in \mathbb{N}$, where $\widetilde{D}_l := Z_lD_lZ_{l+1}^{-1} \in \emph{GL}_n(F_l)$. Further, the isomorphism $\alpha_L$ of part 2. can be extended to an isomorphism 
	\begin{align*}
		\alpha : \emph{Gal}_{\emph{ID}}(E/F) \rightarrow \mathcal{G}(K) \ \ \ \text{with} \ \ \emph{res} \circ \overline{\alpha} = \beta \circ \alpha.
	\end{align*} 
\end{enumerate}
\end{thm} 
\begin{proof}
\cite{Mat01}, Theorem 8.2.
\end{proof}

\begin{remark} \label{rem decomp automorphism}
With assumptions as in Theorem \ref{thm semidirect}, we get the following ID-embedding problem 
\begin{align*}
	\begin{xy}
  	\xymatrix{
        &    & \text{Gal}_{\text{ID}}(E/F) \ar[d]^{\alpha} \ar[r]^{\text{res}} & \text{Gal}_{\text{ID}}(L/F) \ar[d]_{\cong}^{\overline{\alpha}} \ar[ld]_{\chi} & \\
     		1 \ar[r] &  \mathcal{G}^{\circ}(K) \ar[r] & \mathcal{G}^{\circ}(K) \rtimes H \ar[r]^{\beta}  & H \ar[r] & 1. 
  	}
	\end{xy} 
\end{align*} 
One of the key steps in the proof of Theorem \ref{thm semidirect} is that each $\gamma \in \text{Gal}_{\text{ID}}(E/F)$ can be decomposed into $\gamma = \epsilon \circ \eta$, where $\epsilon := \gamma|_{\text{Gal}_{\text{ID}}(E/L)}$ and $\eta := \widetilde{\alpha}^{-1} \circ \chi \circ \text{res}_L(\gamma)$. In particular, every element of $\text{Gal}_{\text{ID}}(L/F)$ can be extended (via the $H$-equivariance) to an element of $\text{Gal}_{\text{ID}}(E/F)$ (see proof of \cite{Mat01}, Theorem 8.2). 
\end{remark}

\begin{remark} \label{rem from} 
With notations as in Theorem \ref{thm semidirect}, $H$ acts on $\mathcal{G}^{\circ}(L)$ via 
\begin{align*}
	\eta \ast g_l := \chi(\eta)\eta(g_l)\chi(\eta)^{-1}, \ \ \ \ \ g_l \in \mathcal{G}^{\circ}(L_l), \eta \in \text{Gal}_{\text{ID}}(L/F). 
\end{align*}
Therefore the $H$-equivariance condition may be reformulated as an invariance condition: 
\begin{align*}
	g_l = \eta \ast g_l \ \ \ \ \ \ \ \text{for all } \eta \in \text{Gal}_{\text{ID}}(L/F) \ \ \ (g_l \in \mathcal{G}^{\circ}(L_l)).
\end{align*} 
The homomorphism $\chi$ defines an element $\chi$ in $H^1(\text{Gal}_{\text{ID}}(L/F),\mathcal{G}^{\circ}(L) \rtimes H)$. Then there is a canonical map $\mathcal{G}^{\circ}(L_l) \rtimes H \rightarrow \text{Aut}(\mathcal{G}^{\circ}(L_l) \rtimes H), g_l \mapsto (h_l \mapsto g_lh_lg_l^{-1})$. The induced map on cohomology maps $\chi$ to an element \\ $\text{Int}(\overline{\alpha}) \in H^1(\text{Gal}_{\text{ID}}(L/F),\text{Aut}_{L}(\mathcal{G} \rtimes H))$. Any automorphism of $\mathcal{G}^{\circ}(L) \rtimes H$ stabilizes the connected component, i.e., we obtain an element in $H^1(\text{Gal}_{\text{ID}}(L/F),\text{Aut}_{L}(\mathcal{G}^{\circ}))$, which is again denoted by $\chi$. We may also define a twisted action as above on the coordinate ring $L[\mathcal{G}]$ by 
\begin{align*}
	(\eta \ast q)(g_l) = \eta(q)(\chi(\eta)^{-1}g_l\chi(\eta)),  \ \ \ \ \ q \in L[\mathcal{G}], g_l \in \mathcal{G}(L_l), 
\end{align*} 
where $\eta(q)$ denotes the Galois action on the coefficients of $q$. Note that this $\ast$-action is semilinear and thus defines an $L/F$-form $\mathcal{G}^{\circ}_{\chi}$ of $\mathcal{G}^{\circ}$, on which the $\ast$-action is the Galois action (see also \cite{Spr98}, 12.3.7). \\ 
By Remark \ref{rem gal iso} we can analogous define an $L_l/F_l$-form $\mathcal{G}^{\circ}_{\chi_l}$ of $\mathcal{G}^{\circ}$ (where $\chi_l = \sigma \circ \overline{\alpha}_l$). But since $\text{Gal}(L_l/F_l)$ and $\text{Gal}(L/F)$ are natural isomorphic, in the following we write the form $\mathcal{G}^{\circ}_{\chi}$ in place of the family of forms $(\mathcal{G}^{\circ}_{\chi_l})_{l \in \mathbb{N}}$. 
\end{remark}

Now we can extend the concept of effectivity to the semidirect product $\mathcal{G}^{\circ} \rtimes H$.

\begin{defn} \label{def H-eff}
Let $H$ be a finite group and let $\mathcal{G} := \mathcal{G}^{\circ} \rtimes H \le \text{GL}_{n,K}$. Let $E/F$ be an IPV-extension and let $L$ be the algebraic closure of $F$ in $E$. The monomorphism $\alpha$$: \text{Gal}_{\text{ID}}(E/F) \rightarrow \mathcal{G}(K)$ is called \textbf{$H$-effective}, if the following conditions hold:
\begin{enumerate}
	\item There exist matrices $Z_l \in \text{GL}_n(L_l)$, such that $\iota_{Z_0}$$: \text{Gal}_{\text{ID}}(L/F) \rightarrow H, \eta \mapsto C_{\eta} := Z_0\eta({Z_0})^{-1}$ is an isomorphism. 
	\item There exist matrices $D_l \in \mathcal{G}(L_l)$, which are $H$-equivariant via $\iota_{Z_0}$, with fundamental solution matrices $Y_l \in \mathcal{G}(E_l)$. 
	\item For $\widetilde{Y}_l := Z_lY_l \in \text{GL}_n(E_l)$, $\alpha = \iota_{\widetilde{Y}_0}$ holds. 
\end{enumerate} 
\end{defn}

For fields of cohomological dimension $\le 1$ we get an analogous result as for connected groups.

\begin{thm} \label{thm cd H-eff}
Let $F$ be an ID-field with $\emph{cd}(F) \le \emph{1}$ and with algebraically closed field of constants $K$. Let $E/F$ be an IPV-extension. Let $H$ be a finite group and let $\mathcal{G} := \mathcal{G}^{\circ} \rtimes H \le \emph{GL}_{n,K}$. Then every isomorphism $\alpha$$: \emph{Gal}_{\emph{ID}}(E/F) \rightarrow \mathcal{G}(K)$ is $H$-effective. 
\end{thm} 
\begin{proof}
The first part of Definition \ref{def H-eff} follows by Proposition \ref{prop semidirect} (1). \\ 
By Proposition \ref{prop semidirect} (2) we have $\widetilde{\eta}(Y_l) = C^{-1}_{\eta}Y_lC_{\eta}$ for all $\eta \in \text{Gal}_{\text{ID}}(L/F)$ and $Y_{l+1} = D^{-1}_lY_l$. Hence the desired property $\eta(D_l) = \eta(Y_l)\eta(Y^{-1}_{l+1}) = C^{-1}_{\eta}D_lC_{\eta}$ holds. 
Finally, by the inclusion 
\begin{align*}
	(\mathcal{G} \rtimes H)(K) \stackrel{\alpha^{-1}}{\rightarrow} \text{Gal}_{\text{ID}}(E/F) \stackrel{\iota_{\widetilde{Y}_0}}{\rightarrow} \text{GL}_n(K),
\end{align*}
we can assume that $\iota_{\widetilde{Y}_0} = \alpha$ holds. 
\end{proof} 

\begin{defn}
With notation as in Definition \ref{def ID-EBP}, we call an ID-embedding problem $\mathcal{E}(\alpha, \beta)$ \textbf{$H$-effective}, if $\alpha$ is $H$-effective. Further we say the solution is \textbf{$H$-effective}, if $\widetilde{\alpha}$ is $H$-effective. 
\end{defn}

\section{Frattini-Embedding Problems} 

\begin{prop} \label{prop frat ep}
Let $F$ be an ID-field with field of constants $K$. Then every solution of a Frattini ID-embedding problem in $\emph{\textbf{AffGr}}_K^{\emph{red}}$ is proper. 
\end{prop}
\begin{proof}
\cite{Mat01}, Proposition 5.15.
\end{proof}

\begin{thm} \label{thm solution frat}
Let $F$ be an ID-field with $\emph{cd}(F) \le 1$. Then every $H$-rigid Frattini ID-embedding problem in $\emph{\textbf{AffGr}}_K^{\emph{red}}$ over $F$ has an $H$-effective, proper solution. 
\end{thm}
\begin{proof} 
By Theorem \ref{thm cd H-eff}, we can assume that the ID-embedding problem is $H$-effective. Let 
\begin{align*}
	\begin{xy}
  	\xymatrix{
        &   &   & \text{Gal}_{\text{ID}}(E/F) \ar[d]_{\cong}^{\alpha} \ar[r]^{\text{res}} & \text{Gal}_{\text{ID}}(L/F) \ar[ld]_{\chi} \\
     		1 \ar[r] &  \mathcal{A}(K) \ar[r] & \widetilde{\mathcal{G}}^{\circ}(K) \rtimes H \ar[r]^{\beta}  & \mathcal{G}^{\circ}(K) \rtimes H \ar[r] & 1.
  	}
	\end{xy} 
\end{align*} 
be an $H$-rigid, $H$-effective Frattini ID-embedding problem, with regular homomorphic section $\sigma$$ : H \rightarrow \mathcal{G}^{\circ}(K) \rtimes H$. For the semidirect product $\mathcal{G}^{\circ}(K) \rtimes H$ we use the notation as in Theorem \ref{thm semidirect}. Since $E/F$ is $H$-effective, the IPV-extension $E/L$ is defined by matrices $D_l \in \mathcal{G}^{\circ}(L_l)$, where $L := E^{\mathcal{G}^{\circ}(K)}$. Let $\chi := \sigma \circ \overline{\alpha} $$: \text{Gal}_{\text{ID}}(L/F) \rightarrow H \rightarrow \mathcal{G}^{\circ}(K) \rtimes H$, then by Remark \ref{rem from}, there exists an $L$-form $\mathcal{G}^{\circ}_{\chi}$ of $\mathcal{G}^{\circ}$ defined over $F$ such that the action given by 
\begin{align*}
	\eta \ast A_l = \chi(\eta)\eta(A_l)\chi(\eta)^{-1}, \ \ \ \ \ A_l \in \mathcal{G}^{\circ}(L_l), \eta \in \text{Gal}_{\text{ID}}(L/F) 
\end{align*} 
is the Galois action on $\mathcal{G}^{\circ}_{\chi}(L)$. We may view the $F_l$-rational points of $\mathcal{G}^{\circ}_{\chi}$ as lying inside $\mathcal{G}^{\circ}(L_l)$, invariant under the action described above. In this formulation, $D_l$ satisfies the equivariance condition if and only if $D_l \in \mathcal{G}^{\circ}_{\chi}(F_l)$ (see again Remark \ref{rem from}). We also can consider the map $\widetilde{\chi} := \widetilde{\sigma} \circ \beta^{-1} \circ \overline{\alpha}$$: \text{Gal}_{\text{ID}}(L/F) \rightarrow H \rightarrow H \rightarrow \widetilde{\mathcal{G}}^{\circ}(K) \rtimes H$, where $\widetilde{\sigma}$$ : H \rightarrow \widetilde{\mathcal{G}}^{\circ}(K) \rtimes H$ is a regular homomorphic section and $\beta|_H = \text{id}_H$. Therefore there exists an $L$-form $\widetilde{\mathcal{G}}^{\circ}_{\widetilde{\chi}}$ of $\widetilde{\mathcal{G}}^{\circ}$ defined over $F$ with the same $\ast$-action as above. The epimorphism $\beta$$: \widetilde{\mathcal{G}}^{\circ}(K) \rtimes H \rightarrow \mathcal{G}^{\circ}(K) \rtimes H$ induces an epimorphism of $L$-forms $\widehat{\beta}$$: \widetilde{\mathcal{G}}^{\circ}_{\widetilde{\chi}} \rightarrow \mathcal{G}^{\circ}_{\chi}$, because $\beta \circ \widetilde{\chi} = \chi \circ \beta$ and $\beta|_H = \text{id}_H$. Given that $D_l$ satisfies the equivariance condition we have $D_l \in \mathcal{G}^{\circ}_{\chi}(F_l)$. Choose preimages $\widetilde{D_l} \in \widehat{\beta}^{-1}(D_l) \subseteq \widetilde{\mathcal{G}}^{\circ}_{\widetilde{\chi}}(F_l)$. Let $\widetilde{E}/L$ be an IPV-extension defined by the matrices $\widetilde{D_l}$, then by \cite{Mat01}, Theorem 5.12, $\widetilde{E} \ge E$ up to an ID-isomorphism. Further, by \cite{Mat01},Theorem 5.1, there exists a monomorphism $\widetilde{\alpha}^{\circ}$$: \text{Gal}_{\text{ID}}(\widetilde{E}/L) \rightarrow \widetilde{\mathcal{G}}^{\circ}(K)$ which is a solution of the corresponding connected ID-embedding problem $\mathcal{E}(\beta^{\circ},\alpha^{\circ})$, where $\beta^{\circ} := \beta|_{\widetilde{\mathcal{G}}^{\circ}}$ and $\alpha^{\circ} := \alpha|_{\text{Gal}_{\text{ID}}(E/L)}$. By Proposition \ref{prop frat ep}, $\widetilde{\alpha}^{\circ}$ is proper, so $\text{Gal}_{\text{ID}}(\widetilde{E}/L) \cong \widetilde{\mathcal{G}}^{\circ}(K)$. Since $\widetilde{D_l} \in \widetilde{\mathcal{G}}^{\circ}_{\widetilde{\chi}}(F_l)$ we have that $\widetilde{D_l}$ are $H$-equivariant. Hence by Theorem \ref{thm semidirect}, $\widetilde{E}/F$ is an IPV-extension with monomorphism $\widetilde{\alpha}$$ : \text{Gal}_{\text{ID}}(\widetilde{E}/F) \rightarrow \widetilde{\mathcal{G}}^{\circ}(K) \rtimes H$. and $\widetilde{\alpha}|_{\text{Gal}_{\text{ID}}(E/F)} = \alpha$. Since $\widetilde{\alpha}^{\circ}$ is an isomorphism, the map $\widetilde{\alpha}$ is also an isomorphism and therefore a $H$-effective, proper solution of the initial embedding problem. 
\end{proof}

\section{ID-Embedding Problems with Finite Kernel} 

By an easy calculation we obtain the following useful lemma. 

\begin{lemma} \label{lem fibre prod}
Let $E/F$, $\widetilde{E}/F$ and $E \cap \widetilde{E}/F$ be IPV-extensions. Then the canonical map $\emph{Gal}_{\emph{ID}}(E \cdot \widetilde{E}/F) \rightarrow \emph{Gal}_{\emph{ID}}(E/F) \times_{\emph{Gal}_{\emph{ID}}(E \cap \widetilde{E}/F)} \emph{Gal}_{\emph{ID}}(\widetilde{E}/F)$ is an isomorphism of groups. Here the product on the right hand side is the fibre product. 
\end{lemma}

\begin{thm} \label{thm solution finite kernel}
Let $F$ be an algebraic function field in one variable over $K$. Then every ID-embedding problem in $\emph{\textbf{AffGr}}_K^{\emph{red}}$ over $F$ with finite kernel has a proper solution. 
\end{thm}
\begin{proof}
Let the following ID-embedding problem with finite kernel be given: 
\begin{align*}
	\begin{xy}
  	\xymatrix{
        &    &   & \text{Gal}_{\text{ID}}(E/F) \ar[d]_{\cong}^{\alpha} & \\
     		1 \ar[r] &  \mathcal{A}(K) \ar[r] & \widetilde{\mathcal{G}}(K) \ar[r]^{\beta}  & \mathcal{G}(K) \ar[r] & 1. 
  	}
	\end{xy} 
\end{align*} 
If $\widetilde{\mathcal{G}}$ is connected, this ID-embedding problem is a Frattini ID-embedding problem by Lemma \ref{lem frat class}. Moreover this embedding problem is $H$-rigid where in this situation the group $H$ is trivial. Otherwise if $\widetilde{\mathcal{G}}$ is non-connected, $\beta$ is decomposable into a split epimorphism and a Frattini-epimorphism (see Lemma \ref{lem decom2}). Further by Lemma \ref{lem H-semi frat decom}, this Frattini-epimorphism is decomposable into a Frattini-epimorphism of type $\mu$ (see Definition \ref{defn mu}) and an $H$-rigid Frattini-epimorphism. So we have to consider three cases: \\
\emph{1. $\beta$ is an $H$-rigid Frattini-epimorphism:} \\ 
Since $\text{cd}(F) \le 1$ these ID-embedding problems have an $H$-effective, proper solution by Proposition \ref{thm solution frat}. \\ 
\emph{2. $\beta$ is a Frattini-epimorphism of type $\mu$:} \\
Let 
\begin{align*}
	\begin{xy}
  	\xymatrix{
        &    &   & \text{Gal}_{\text{ID}}(E/F) \ar[d]_{\cong}^{\alpha} & \\
     		1 \ar[r] &  \mathcal{G}^{\circ}(K) \cap H \ar[r] & \mathcal{G}^{\circ}(K) \rtimes H \ar[r]^{\beta}  & \mathcal{G}^{\circ}(K) \cdot H \ar[r] & 1 
  	}
	\end{xy} 
\end{align*}
be an ID-embedding problem of type $\mu$. Since $\mathcal{G}^{\circ}$ is a normal subgroup of $\mathcal{G}^{\circ} \cdot H$, there exists an IPV-extension $L/F$, such that $\text{Gal}_{\text{ID}}(E/L) \cong \mathcal{G}^{\circ}(K)$. Therefore, $\text{Gal}_{\text{ID}}(L/F) \cong H/(H \cap \mathcal{G}^{\circ}(K))$ and this leads to a finite ID-embedding problem of the form 
\begin{align*}
	\begin{xy}
  	\xymatrix{
       & & \text{Gal}_{\text{ID}}(\widetilde{L}/F) \ar[r]^{\text{res}} \ar[d]_{\widetilde{\vartheta}} & \text{Gal}_{\text{ID}}(L/F) \ar[d]_{\cong}^{\vartheta} &\\
     		1 \ar[r] &  \mathcal{G}^{\circ}(K) \cap H \ar[r] &  H \ar[r]^/-.5cm/{\overline{\beta}}  & H/(H \cap \mathcal{G}^{\circ}(K)) \ar[r] & 1. 
  	}
	\end{xy} 
\end{align*} 
Since $\pi^{\text{alg}}_1(F)$ is free (\cite{MM99}, Corollary V.2.11), this embedding problem has a proper solution $\widetilde{\vartheta}$$: \text{Gal}_{\text{ID}}(\widetilde{L}/F) \stackrel{\cong}{\rightarrow} H$, where without loss of generality it can be assumed that $\widetilde{L} \cap E = L$. We take the composite $\widetilde{E} := E \cdot \widetilde{L}$, and $\widetilde{E}/F$ is an IPV-extension. By Lemma \ref{lem fibre prod}, we see that $\text{Gal}_{\text{ID}}(\widetilde{E}/F) \cong \text{Gal}_{\text{ID}}(E/F) \times_{\text{Gal}_{\text{ID}}(L/F)} \text{Gal}_{\text{ID}}(\widetilde{L}/F)$ and hence 
\begin{align*}
   \text{Gal}_{\text{ID}}(\widetilde{E}/F) & \cong \text{Gal}_{\text{ID}}(E/F) \times_{\text{Gal}_{\text{ID}}(L/F)} \text{Gal}_{\text{ID}}(\widetilde{L}/F) \\
   & \cong \mathcal{G}^{\circ}(K) \cdot H \times_{H/(H \cap \mathcal{G}^{\circ}(K))} H  \cong \mathcal{G}^{\circ}(K) \rtimes H 
\end{align*} 
is a proper solution of the initial ID-embedding problem. \\
\emph{3. $\beta$ is a split epimorphism:} \\
Thanks to case 2, we can assume that $\mathcal{G} \cong \mathcal{G}^{\circ} \rtimes H$ with finite group $H$. Therefore we have to solve the following split ID-embedding problem: 
\begin{align*} \tag{$\ast$}
	\begin{xy}
  	\xymatrix{
        &    &   & \text{Gal}_{\text{ID}}(E/F) \ar[d]_{\cong}^{\alpha} & \\
     		1 \ar[r] &  \mathcal{A} \ar[r] & \mathcal{A} \rtimes (\mathcal{G}^{\circ} \rtimes H)(K) \ar[r]^/0.2cm/{\beta}  & (\mathcal{G}^{\circ} \rtimes H)(K) \ar[r] & 1 
  	}
	\end{xy} 
\end{align*} 
Since $\mathcal{A}$ is a finite group, $\mathcal{G}^{\circ}$ acts trivially on $\mathcal{A}$. So $H$ acts on $\mathcal{A}$ and this induces the following finite, split ID-embedding problem: 
\begin{align*}
	\begin{xy}
  	\xymatrix{
        &    &   & \text{Gal}_{\text{ID}}(L/F) \ar[d]_{\cong}^{\alpha} & \\
     		1 \ar[r] &  \mathcal{A} \ar[r] & \mathcal{A} \rtimes H \ar[r]^/0.2cm/{\beta}  &  H \ar[r] & 1, 
  	}
	\end{xy} 
\end{align*} 
where $L := E^{\mathcal{G}^{\circ}(K)}$. This ID-embedding problem has a proper solution $\widetilde{L}$, since $\pi^{\text{alg}}_1(F)$ is free (see \cite{MM99}, Corollary V.2.11). Therefore, we can take the composite $\widetilde{E} := E \cdot \widetilde{L}$, and 
\begin{align*}
	\text{Gal}_{\text{ID}}(\widetilde{E}/F) & \cong \mathcal{G}^{\circ}(K) \rtimes ( \mathcal{A} \rtimes H) \cong (\mathcal{G}^{\circ}(K) \times  \mathcal{A}) \rtimes H \cong (\mathcal{A} \times \mathcal{G}^{\circ}(K)) \rtimes H \\
	 & \cong \mathcal{A} \rtimes \mathcal{G}^{\circ}(K) \rtimes H
\end{align*} 
is a proper solution of the ID-embedding problem $(\ast)$. 
\end{proof}

\section{ID-Embedding Problems with Unipotent Kernel} 

The next theorem is an extension of Theorem \ref{thm semidirect}. 

\begin{thm} \label{thm semidirect 2}
Let $F$ be an ID-field with field of constants $K$ and $\emph{cd}(F) \le 1$. Let $H$ be a finite group, $\mathcal{G}^{\circ}$ be an $H$-group with semidirect product $\mathcal{G} := \mathcal{G}^{\circ} \rtimes H \in \emph{\textbf{AffGr}}_K^{\emph{red}}$ and let $\mathcal{U} \in \emph{\textbf{AffGr}}_K^{\emph{red}}$ be a connected group with semidirect product $\widetilde{\mathcal{G}} := \mathcal{U} \rtimes \mathcal{G} \in \emph{\textbf{AffGr}}_K^{\emph{red}}$. 
\begin{enumerate} 
 \item Let $E/F$ be an IPV-extension with isomorphim $\alpha$$: \emph{Gal}_{\emph{ID}}(E/F) \stackrel{\cong}\rightarrow \mathcal{G}(K), \\ \gamma \mapsto C_{\gamma}$ and let $L := E^{\mathcal{G}^{\circ}(K)}$. Then for all $l \in \mathbb{N}$ there exist matrices $Z_l \in \emph{GL}_n(L_l)$ and $Y_l \in \mathcal{G}^{\circ}(E_l)$ such that $\widetilde{Y}_l = Z_lY_l$ satisfying $E = F(\widetilde{Y}_0)$ and $\gamma(\widetilde{Y}_l) = \widetilde{Y}_lC_{\gamma}$ for all $\gamma \in \emph{Gal}_{\emph{ID}}(E/F)$. 
\item Let $U_l \in \mathcal{U}(E_l)$ be $\mathcal{G}^{\circ} \rtimes H$-equivariant via $\alpha$, i.e. $\gamma(U_l) = C^{-1}_{\gamma}U_lC_{\gamma}$ for all $\gamma \in \emph{Gal}_{\emph{ID}}(E/F)$. Then $U_l$ define an IPV-extension $\widetilde{E}/E$ with fundamental solution matrices $X_l \in \mathcal{U}(\widetilde{E}_l)$ and monomorphism $\widetilde{\alpha}_E$$: \emph{Gal}_{\emph{ID}}(\widetilde{E}/E) \rightarrow \mathcal{U}(K)$. 
\item Further $\widetilde{E}/F$ is an IPV-extension with fundamental solution matrices $\widetilde{X}_l = \widetilde{Y}_lX_l$ and monomorphism $\widetilde{\alpha}$$: \emph{Gal}_{\emph{ID}}(\widetilde{E}/F) \rightarrow \widetilde{\mathcal{G}}(K)$, such that \newline $\widetilde{\alpha}|_{\emph{Gal}_{\emph{ID}}(E/F)} = \alpha$ and $\widetilde{\alpha}|_{\emph{Gal}_{\emph{ID}}(\widetilde{E}/E)} = \widetilde{\alpha}_E$. 
\end{enumerate}
\end{thm} 
\begin{proof}
With notation as in Theorem \ref{thm semidirect} we consider the following ID-embedding problem: 
\begin{align*}
	\begin{xy}
  	\xymatrix{
        &    & \text{Gal}_{\text{ID}}(\widetilde{E}/F) \ar[r]^{\text{res}} \ar[d]^{\widetilde{\alpha}}  & \text{Gal}_{\text{ID}}(E/F)  \ar[r]^{\text{res}} \ar[d]_{\cong}^{\alpha} & \text{Gal}_{\text{ID}}(L/F) \ar[ld]_{\chi} \\
     		1 \ar[r] &  \mathcal{U}(K) \ar[r] & \mathcal{U}(K) \rtimes (\mathcal{G}^{\circ}(K) \rtimes H) \ar[r]^/0.3cm/{\widetilde{\beta}}  & \mathcal{G}^{\circ}(K) \rtimes H \ar[r] \ar@<4pt>^/-0.3cm/{\widetilde{\sigma}}[l] & 1.  
  	}
	\end{xy} 
\end{align*} 
Remark that the upper row in this diagram is not an exact sequence. \\
\emph{1.} The first claim follows by Theorem \ref{thm semidirect}. For the second claim we decompose $\gamma \in \text{Gal}(E/F)$ into $\epsilon \circ \eta$, as in Remark \ref{rem decomp automorphism}. Then $\gamma(\widetilde{Y}_l) = \eta(Z_l)\epsilon(\eta(Y_l)) = Z_lC_{\eta}\epsilon(C^{-1}_{\eta}Y_lC_{\eta}) = Z_l\epsilon(Y_l)C_{\eta} = Z_lY_lC_{\epsilon}C_{\eta} = \widetilde{Y}_lC_{\gamma}$. \\ 
\emph{2.} Since $U_l \in \mathcal{U}(E_l)$, the IPV-extension $\widetilde{E}/E$ has fundamental solution matrices $X_l \in \mathcal{U}(\widetilde{E}_l)$ with $X_{l+1} = U_lX_l$ and hence we get a monomorphism $\widetilde{\alpha}_{\mathcal{U}}$$: \text{Gal}_{\text{ID}}(\widetilde{E}/E) \rightarrow \mathcal{U}(K)$. \\
\emph{3.} With notation as in Theorem \ref{thm semidirect} we have that $Y_{l+1} = D^{-1}_lY_l$. Let $\widehat{U_l} := Y_lU_lY^{-1}_{l+1}$, then $\gamma(\widehat{U}_l) = \gamma(Y_lU_lY^{-1}_{l+1}) = Y_lC_{\gamma}C^{-1}_{\gamma}U_lC_{\gamma}C^{-1}_{\gamma}Y^{-1}_{l+1} = \widehat{U}_l$ for all $\gamma \in \text{Gal}(E/L)$ and therefore $\widehat{U}_l \in \widetilde{\mathcal{G}}^{\circ}(L_l)$. A quick calculation shows that $\widehat{X}_l := Y_lX_l \in (\mathcal{U} \rtimes \mathcal{G}^{\circ})(\widetilde{E}_l)$ are fundamental solution matrices for $\widehat{U}_l$. Given that $\beta(\widehat{U}_l) = D_l$, we get by \cite{Mat01}, Theorem 5.12, that $\widetilde{E} \ge L(\widehat{X}) \ge E$ up to an ID-isomorphism, where $\widehat{X} := \widehat{X}_0$. Since $X_l = Y^{-1}_l\widehat{X}_l$, we have $L(\widehat{X}) = E(\widehat{X}) = \widetilde{E}$. \\ 
Let $\widetilde{\eta} := \alpha^{-1} \circ \chi(\eta) \in \text{Gal}_{\text{ID}}(E/F)$ be a preimage of $\eta$, then for all $\eta \in \text{Gal}_{\text{ID}}(L/F)$ the following equation holds 
\begin{align*}
	 \eta(\widehat{U}_l) & = \widetilde{\eta}(\widehat{U}_l) = \widetilde{\eta}(Y_lU_lY^{-1}_{l+1}) = C^{-1}_{\eta}Y_lC_{\eta}C^{-1}_{\widetilde{\eta}}U_lC_{\widetilde{\eta}}C^{-1}_{\eta}Y^{-1}_{l+1}C_{\eta} \\ & = C^{-1}_{\eta}Y_lU_lY^{-1}_{l+1}C_{\eta} = C^{-1}_{\eta}\widehat{U}_lC_{\eta}. 
\end{align*}
That is $\widehat{U}_l$ are $H$-equivariant and so the claim follows by Theorem \ref{thm semidirect}. 
\end{proof}

\begin{lemma} \label{lem unipot equi}
With assumptions as in Theorem \ref{thm semidirect 2} we have: \\ 
The matrices $U_l$ are $\emph{Gal}(E/F)$-equivariant if and only if $U_l \in \widetilde{Y}^{-1}_l\mathcal{U}_{\chi}(F_l)\widetilde{Y}_l$. 
\end{lemma}
\begin{proof}
Let $U_l \in \mathcal{U}(E_l)$ be $\text{Gal}(E/F)$-equivariant and let $\widetilde{U}_l := \widetilde{Y}_lU_l\widetilde{Y}^{-1}_l \in \mathcal{U}(E_l)$. By the notations from Theorem \ref{thm semidirect 2} (i), we have that $\widetilde{Y}_l = Z^{-1}_lY_l$ and for all $\gamma = \epsilon \circ \eta \in \text{Gal}(E/F)$ the following equation holds: 
\begin{align*}
	\gamma(\widetilde{U}_l) & = \gamma(Z^{-1}_lY_lU_lY^{-1}_lZ_l) = \eta(Z^{-1}_l)Y_lC_{\gamma}C^{-1}_{\gamma}U_lC_{\gamma}C^{-1}_{\gamma}Y^{-1}_l\eta(Z_l) \\
	& = C^{-1}_{\eta}Z^{-1}_lY_lU_lY^{-1}_lZ_lC_{\eta} = C^{-1}_{\eta}\widetilde{U}_lC_{\eta} = \eta(\widetilde{U}_l) \in \mathcal{U}(L_l). 
\end{align*} 
Therefore the matrices $U_l$ are $\text{Gal}(E/F)$-equivariant if and only if $\widetilde{U}_l \in \mathcal{U}_{\chi}(F_l)$. 
\end{proof}

By \cite{Roe07}, Theorem 9.11, every ``connected`` ID-embedding problems with unipotent kernel has a proper solution. With the work we have done above, we can extend this theorem to the case of $H$-rigid ID-embedding problems. The proof is up to some small details the same as in \cite{Roe07}. 

\begin{thm} \label{thm solution unipot kernel}
Let $F$ be an ID-field. Then every $H$-rigid, $H$-split, $H$-effective ID-embedding problem in $\emph{\textbf{AffGr}}_K^{\emph{red}}$ over $F$ with minimal, reduced, connected, unipotent kernel has an $H$-effective, proper solution. 
\end{thm}
\begin{proof} 
Let  
\begin{align*}
	\begin{xy}
  	\xymatrix{
        &    & \text{Gal}_{\text{ID}}(\widetilde{E}/F) \ar[r]^{\text{res}} \ar[d]^{\widetilde{\alpha}}  & \text{Gal}_{\text{ID}}(E/F)  \ar[r]^{\text{res}} \ar[d]_{\cong}^{\alpha} &  \\
     		1 \ar[r] &  \mathcal{U}(K) \ar[r] & \mathcal{U}(K) \rtimes (\mathcal{G}^{\circ}(K) \rtimes H) \ar[r]^/0.3cm/{\widetilde{\beta}}  & \mathcal{G}^{\circ}(K) \rtimes H \ar[r] \ar@<4pt>^/-0.3cm/{\widetilde{\sigma}}[l] & 1.  
  	}
	\end{xy} 
\end{align*} 
be an $H$-rigid, $H$-split, $H$-effective ID-embedding problem over $F$ with minimal, reduced, connected, unipotent kernel. We use the notation as in Theorem \ref{thm semidirect 2}. \\ 
The group $\mathcal{G}^{\circ} \rtimes H$ acts on $\mathcal {U}$ and the center $\text{Z}(\mathcal{U})$ is a non-trivial, characteristic subgroup of $\mathcal{U}$, so it is $\mathcal{G}^{\circ} \rtimes H$-invariant and a normal subgroup of $\mathcal{U}$. Since $\mathcal{U}$ is unipotent, $\text{Z}(\mathcal{U})$ is nontrivial and therefore by minimality of $\mathcal{U}$, we get $\text{Z}(\mathcal{U}) = \mathcal{U}$, i.e. $\mathcal{U}$ is abelian. \\ 
Further if $\mathcal{A}$ is a non-connected, $\mathcal{G}^{\circ} \rtimes H$-invariant, normal subgroup of $\mathcal{U}$, then $\mathcal{A}$ is finite, because its identity component $\mathcal{A}^{\circ}$ is also $\mathcal{G}^{\circ} \rtimes H$-invariant and normal in $\mathcal{U}$ and hence trivial by minimality of $\mathcal{U}$. \\
By Theorem \ref{thm semidirect 2}, every sequence $U_l \in \mathcal{U}(E_l)^{\mathcal{G}^{\circ} \rtimes H} := \{U_l \in \mathcal{U}(E_l) | U_l \ \text{are} \ \mathcal{G}^{\circ} \rtimes H\text{-equivariant}\}$ $(l \in \mathbb{N})$ defines an IPV-extension $\widetilde{E}/F$ with $\text{Gal}_{\text{ID}}(\widetilde{E}/E) \leq \mathcal{U}(K)$ and $\text{Gal}_{\text{ID}}(\widetilde{E}/F) \leq \mathcal{U}(K) \rtimes (\mathcal{G}^{\circ}(K) \rtimes H)$. Since $\mathcal{U}$ is minimal, we obtain that $\text{Gal}_{\text{ID}}(\widetilde{E}/E)$ is finite or $\text{Gal}_{\text{ID}}(\widetilde{E}/E) = \mathcal{U}(K)$. So we have to show that there exists a sequence $U_l \in \mathcal{U}(E_l)^{\mathcal{G}^{\circ} \rtimes H}$ such that $\widetilde{E}/E$ is not finite. \\  
By Lemma \ref{lem unipot equi}, every sequence $\widetilde{U}_l \in \mathcal{U}_{\chi}(F_l)$ define fundamental solution matrices $\widetilde{X}_l$ with $\widetilde{X}_{l+1} = \widetilde{U}_l\widetilde{X}^{-1}_l$. Two such sequences $(\widetilde{U}_l)_{l \in \mathbb{N}}$, $(\widetilde{U}^{'}_l)_{l \in \mathbb{N}}$ have fundamental solution matrices $\widetilde{X}_l$, $\widetilde{X}^{'}_l$ which are ID-isomorphic over $F$ if and only if $(\widetilde{U}_0 \cdots \widetilde{U}_l)^{-1}(\widetilde{U}^{'}_0 \cdots \widetilde{U}^{'}_l) \in \mathcal{U}_{\chi}(F_{l+1})$. 
Therefore we have a one-to-one correspondence between ID-isomorphism classes of those fundamental solution matrices and the infinite product 
\begin{align*}
	\varprojlim (\mathcal{U}_{\chi}(F)/\mathcal{U}_{\chi}(F_{l})) = \prod_{l\in\mathbb{N}}{\mathcal{U}_{\chi}(F_{l})/\mathcal{U}_{\chi}(F_{l+1})} .
\end{align*} 
Since $\mathcal{U} \cong \mathbb{G}^m_a$ and $H^1(\text{Gal}_{\text{ID}}(L/F),\mathbb{G}^m_a) = 1$ (see \cite{Spr98}, Example 12.3.5), by \cite{Spr98}, Proposition 12.3.2, we obtain that $\mathcal{U}_{\chi}(F_{l}) \cong \mathcal{U}(F_l)$ as $F_l$-vector spaces. Thus $\mathcal{U}_{\chi}(F_{l})/\mathcal{U}_{\chi}(F_{l+1})$ is a $K$-vector space with \\ $\text{dim}_K(\mathcal{U}_{\chi}(F_{l})/\mathcal{U}_{\chi}(F_{l+1})) \geq \text{dim}_K(F_l/F_{l+1}) \geq 2$. Hence the dimension of the infinite product as $K$-vector space is uncountable $(\geq 2^{\mathbb{N}})$. \\ 
Those fundamental solution matrices whose IPV-extension is finite are given by maximal ideals in the ring $E[\widetilde{X}_{ij},\text{det}(\widetilde{X})^{-1}]$. Every maximal ideal is given by $n^2$ polynomials, thus the $E$-vector space $V$ of $n^2$-tuples of polynomials gives an upper bound to the number of those fundamental solution matrices with finite IPV-extension. But since $V$ is an $E$-vector space of countable dimension and $E$ is a $K$-vector space of countable dimension, $V$ is a $K$-vector space of countable dimension. \\
Thus for dimensional reasons, there exists a sequence $U_l \in \mathcal{U}(E_l)$ with IPV-extension $\widetilde{E}$, such that $\text{Gal}_{\text{ID}}(\widetilde{E}/F) \cong \mathcal{U}(K) \rtimes \mathcal{G}^{\circ}(K) \rtimes H$ and $\widetilde{E}/F$ is $H$-effective by construction. 
\end{proof}

\section{The Embedding Theorem} 

The next lemma is taken from \cite{Obe03}. 

\begin{lemma} \label{lem direct H-split}
Let $H$ be a finite group and $\widetilde{\mathcal{G}} := \widetilde{\mathcal{G}}^{\circ} \rtimes H, \mathcal{G} := \mathcal{G}^{\circ} \rtimes H \in \emph{\textbf{AffGr}}_K^{\emph{red}}$, with $H$-rigid, $H$-split epimorphism $\beta$$: \widetilde{\mathcal{G}}^{\circ} \rtimes H \rightarrow \mathcal{G}^{\circ} \rtimes H$. 
\begin{enumerate} 
 \item If \emph{ker($\beta$)} is a reduced torus $\mathcal{T}$, then $\beta$ is subdirect $H$-split, i.e., $\widetilde{\mathcal{G}}^{\circ} \cong  \mathcal{T} \times \mathcal{G}^{\circ}$. 
 \item If \emph{ker($\beta$)} is a reduced, semi-simple, centerless group $\mathcal{A}$, then $\beta$ is subdirect $H$-split, i.e., $\widetilde{\mathcal{G}}^{\circ} \cong \mathcal{A} \times \mathcal{G}^{\circ}$. 
\end{enumerate}
\end{lemma}
\begin{proof} 
 \emph{1.} Since $\widetilde{\mathcal{G}} = \text{N}_{\widetilde{\mathcal{G}}}(\mathcal{T})$, we have $\widetilde{\mathcal{G}}^{\circ} = \text{N}_{\widetilde{\mathcal{G}}}(\mathcal{T})^{\circ} = \text{C}_{\widetilde{\mathcal{G}}}(\mathcal{T})^{\circ}$, by \cite{Spr98}, Corollary 3.2.9 and therefore $\mathcal{T} \leq \text{Z}(\widetilde{\mathcal{G}}^{\circ})$. Then via the homomorphic section $\mathcal{G}^{\circ}$ will be a subgroup of $\widetilde{\mathcal{G}}^{\circ}$ and hence $[\mathcal{T}, \mathcal{G}^{\circ}] = 1$. \\
\emph{2.} Immediately from the assumptions we obtain $\mathcal{A} \cap \text{C}_{\widetilde{\mathcal{G}}^{\circ}}(\mathcal{A}) = \text{Z}(\mathcal{A}) = 1$. The group $\widetilde{\mathcal{G}}^{\circ}$ acts on $\mathcal{A}$ via conjugation as automorphism and $\mathcal{A} \cdot \text{C}_{\widetilde{\mathcal{G}}^{\circ}}(\mathcal{A})$ acts also on $\mathcal{A}$ via conjugation as inner automorphism, so $\widetilde{\mathcal{G}}^{\circ}/\mathcal{A} \cdot \text{C}_{\widetilde{\mathcal{G}}^{\circ}}(\mathcal{A}) \leq \text{Aut}(\mathcal{A})/\text{Inn}(\mathcal{A})$ is a finite group by \cite{Hum98}, Theorem 27.4. Given that $\widetilde{\mathcal{G}}^{\circ}$ is connected, we obtain $\mathcal{A} \cdot \text{C}_{\widetilde{\mathcal{G}}^{\circ}}(\mathcal{A}) = \widetilde{\mathcal{G}}^{\circ} = \mathcal{A} \times \text{C}_{\widetilde{\mathcal{G}}^{\circ}}(\mathcal{A})$. Since $\beta$ is $H$-rigid, $H$ acts on $\text{C}_{\widetilde{\mathcal{G}}^{\circ}}(\mathcal{A})$ via conjugation in the same way as on $\widetilde{\mathcal{G}}^{\circ}$, hence $\sigma$$: \mathcal{G}^{\circ} \rtimes H \rightarrow \text{C}_{\widetilde{\mathcal{G}}^{\circ}}(\mathcal{A}) \rtimes H$ is a regular, $H$-rigid homomorphic section. 
\end{proof}

\begin{prop} \label{prop direct H-split solution}
Let $F$ be an algebraic function field in one variable over $K$ and let $H$ be a finite group. Then every $H$-rigid, subdirect $H$-split ID-embedding problem in $\emph{\textbf{AffGr}}_K^{\emph{red}}$ over $F$ with reduced, connected kernel has a proper solution. 
\end{prop}
\begin{proof} 
By Theorem \ref{thm cd H-eff}, we can assume that all such embedding problems are $H$-effective. Let 
\begin{align*}
	\begin{xy}
  	\xymatrix{
        &   &   & \text{Gal}_{\text{ID}}(E/F) \ar[d]_{\cong}^{\alpha} & \\
     		1 \ar[r] &  \mathcal{A}(K) \ar[r] & (\mathcal{A} \times \mathcal{G}^{\circ})(K) \rtimes H \ar[r]^{\beta}  & \mathcal{G}^{\circ}(K) \rtimes H \ar[r] & 1 
  	}
	\end{xy} 
\end{align*} 
be an $H$-rigid, subdirect $H$-split, $H$-effective ID-embedding problem with connected kernel and let $L := E^{\mathcal{G}^{\circ}(K)}$. By the assumptions $H$ acts on $\mathcal{A}$ and therefore $H$ acts also on $\mathcal{A}^{r}$, where $r \in \mathbb{N}$. With $\overline{\alpha}$$: \text{Gal}(L/F) \stackrel{\cong}\rightarrow H$ and $\overline{\beta}_r$$: \mathcal{A}^r \rtimes H \rightarrow H$ (projection on the second factor), we get an ID-embedding problem $\mathcal{E}(\overline{\alpha}, \overline{\beta}_r)$ with connected kernel and finite cokernel. Such an ID-embedding problem has a proper solution by \cite{Mat01}, Corollary 8.9 and if we choose $r > \text{dim}_K(\mathcal{G}^{\circ})$, there exists a fixed field $\widetilde{L}$ of $\mathcal{A}^{r-1}(K) \le \mathcal{A}^r(K)$, with $\widetilde{L} \cap E = L$. Hence $\text{Gal}(\widetilde{L}/F) \cong \mathcal{A}(K) \rtimes H$ and with $\widetilde{E} := \widetilde{L} \cdot E$ we get by Lemma \ref{lem fibre prod} that $\text{Gal}(\widetilde{E}/F) \cong (\mathcal{A}(K) \rtimes H) \times_{H} \mathcal{G}(K) \cong (\mathcal{A} \times \mathcal{G}^{\circ})(K) \rtimes H$.
\end{proof}

\begin{thm} \label{embedding thm}
Let $F$ be an algebraic function field in one variable over an algebraically closed field $K$ of positive characteristic. Then every ID-embedding probelm in $\emph{\textbf{AffGr}}_K^{\emph{red}}$ over $F$, with reduced kernel has a proper solution.
\end{thm}
\begin{proof}
By Propostions \ref{prop emb} and \ref{prop decompostion} we have to show that ID-embedding problems of types (1) - (5) have a proper solution. For type (1) see Theorem \ref{thm solution finite kernel}, type (2) is done by Theorem \ref{thm solution frat}, types (3) and (4) are compleded by Lemma \ref{lem direct H-split} together with Proposition \ref{prop direct H-split solution}, and type (5) follows by Theorem \ref{thm solution unipot kernel}. 
\end{proof}

From the above theorem we obtain a result about the structure of the Tannakian fundamental group for $\textbf{IDMod}_F$.

\begin{defn} \label{red free}
A prolinear group $\pi$ over $K$ is called \textbf{reduced free}, if every ID-embedding problem 
\begin{align*}
	\begin{xy}
  	\xymatrix{
        &    &   & \pi \ar[dl]_{\widetilde{\alpha}} \ar[d]^{\alpha} & \\
     		1 \ar[r] &  \mathcal{A}(K) \ar[r] & \widetilde{\mathcal{G}}(K) \ar[r]^{\beta}  & \mathcal{G}(K) \ar[r] & 1, 
  	}
	\end{xy} 
\end{align*} 
with $\mathcal{A}, \widetilde{\mathcal{G}}, \mathcal{G} \in \text{\textbf{AffGr}}_K^{\text{red}}$ and epimorphisms $\widetilde{\alpha}, \alpha$, has a proper solution. 
\end{defn}

In usual Galois theory by Iwasawa's Freiheitssatz \cite{MM99}, Theorem IV.1.12, the absolute Galois group $\text{Gal}(F^{\text{sep}}/F)$ is free (of countable infinite rank) if and only if all embedding problems over $F$ have a proper solution. With the above definition, the main result, Theorem \ref{embedding thm}, can be seen as a generalization of this result (remember that the usual Galois theory is a special case of the iterative differential Galois theory; see Theorem \ref{finite gal extension}).  

\begin{remark} 
Let $F$ be an algebraic function field in one variable over an algebraically closed field $K$ of positive characteristic. Let $\pi_1(\textbf{IDMod}_F)$ be the Tannakian fundamental group for $\textbf{IDMod}_F$. Then $\pi_1(\textbf{IDMod}_F)$ is an analogon to the absolute Galois group, as in the introduction explained and therefore by Theorem \ref{embedding thm}, $\pi_1(\textbf{IDMod}_F)$ is a reduced free group scheme. 
\end{remark}


\begin{thebibliography}{99}
	\bibitem[Bor91]{Bor91} Borel, A.: \textit{Linear Algebraic Groups, Second Enlarged Edition}. Graduate Texts in Mathematics, Springer Verlag, Berlin Heidelberg New York (1991) 
	\bibitem[BS64]{BS64} Borel, A.. Serre, J.-P.: \textit{Th\'eor\`emes de finitude en cohomologie galoisienne}. Comment. Math. Helv. 39:111-164 (1964) 
	\bibitem[Del90]{Del90} Deligne, P.: \textit{Cat\'egories Tannakiennes}, In \textit{The Grothendieck Festschrift, Collect. Artic. in Honor of the 60th Birthday of A. Grothendieck}. Vol. II, pages 111-195. Birkh\"auser, Boston (1990) 
%	\bibitem[FJ86]{FJ86} Fried, M.D.. Jarden, M.: \textit{Field Arithmetic}. Springer Verlag, Berlin Heidelberg New York (1986) 
%	\bibitem[Har02]{Har02} Hartmann, J.: \textit{On the inverse problem in differential Galois theory}. Dissertation, Heidelberg (2002) (available at http://www.ub.uni-heidelberg.de/archiv/3085/) 
	\bibitem[HS37]{HS37} Hasse H.. Schmidt F.K.: \textit{Noch eine Begr\"undung der Theorie des h\"oheren Differentialquotienten in einem algebraischen Funktionenk\"orper in einer Un\-be\-stimmten}. J. Reine Angew. Math. 177:215-237 (1937) 
	\bibitem[Hum98]{Hum98} Humphreys, J.E.: \textit{Linear Algebraic Groups}. Graduate Texts in Mathematics, Springer Verlag, Berlin Heidelberg New York (1981) 
	\bibitem[Hup67]{Hup67} Huppert, B.: \textit{Endliche Gruppen I}. Grundlehren der mathematischen
Wissenschaften 134, Springer Verlag, Berlin Heidelberg New York (1983)
%	\bibitem[Lan93]{Lan93} Lang, S.: \textit{Algebra}. Addison-Wesley Publishing Company, third edition (1993) 
	\bibitem[MM99]{MM99} Malle, G.. Matzat, B.H.: \textit{Inverse Galois Theory}. Springer-Verlag, Berlin Heidelberg (1999) 
	\bibitem[Mat01]{Mat01} Matzat, B.H.: \textit{Differential Galois Theory in Positive Characteristic}, notes written by J. Hartmann. IWR-Preprint 2001-35 (2001) (available at http://www.iwr.uni-heidelberg.de/\~{}Heinrich.Matzat/publications.html)
	\bibitem[MvdP03]{MvdP03} Matzat, B.H.. van der Put, M.: \textit{Constructive Differential Galois Theory}. Galois Groups and Fundamental Groups, MSRI-Publications 41, 425-467 (2003)
%	\bibitem[Mau10a]{Mau10a} Maurischat, A.: \textit{Galois theory for iterative connections and nonreduced Galois groups}, to appear in Transactions of the AMS, 362 (2010), no. 10, pp. 5411-5453 (preprint available from arXiv at http://arxiv.org/abs/0712.3748) 
%	\bibitem[Mau10b]{Mau10b} Maurischat, A.: \textit{Infinitesimal group schemes as iterative differential galois groups}, to appear in Journal of Pure and Applied Algebra 214 (2010), no. 11, pp. 2092-2100 (preprint available from arXiv at http://arxiv.org/abs/0908.3772)
	\bibitem[Obe03]{Obe03} Oberlies, T.: \textit{Einbettungsprobleme in der Differentialgaloistheorie}. Dissertation, Heidelberg (2003) (available at http://www.ub.uni-heidelberg.de/archiv/4550/) 
	\bibitem[Roe07]{Roe07} R\"oscheisen, A.: \textit{Iterative Connections and Abhyankar's Conjecture}. Dissertation, Heidelberg (2007) (available at http://www.ub.uni-heidelberg.de/archiv/7179/) 
	\bibitem[Ser97]{Ser97} Serre, J.-P.: \textit{Galois Cohomology}, Springer Verlag, Berlin Heidelberg New York (1997)
	\bibitem[Spr98]{Spr98} Springer T.A.: \textit{Linear Algebraic Groups, Second Edition}. Birkh\"auser (1998) 
%  \bibitem[Tit68]{Tit68} Tits, J.: \textit{Lectures on Algebraic Groups}. Lecture Notes. Yale University (1968) 
%  \bibitem[SvdP03]{SvdP03} Singer, M.F.. van der Put, M.: \textit{Galois theory of linear differential equations}. Springer Verlag, Berlin Heidelberg New York (2003) 
\end{thebibliography}
\end{document}